\documentclass[10pt,a4paper]{scrartcl} 
\usepackage{amsmath,amssymb}
\usepackage{enumerate}
\usepackage{todonotes}
\usepackage{accents}
\allowdisplaybreaks
\usepackage[left=3cm,right=2cm,top=3cm,bottom=3cm]{geometry}

\usepackage[english]{babel}
\usepackage[T1]{fontenc}
\usepackage[utf8]{inputenc}  
\usepackage{filecontents}
  
\usepackage{amsthm}
\swapnumbers
\newtheorem{satz}{Theorem}[section]
\newtheorem{lemma}[satz]{Lemma}
\newtheorem{kor}[satz]{Corollary}
\theoremstyle{definition}
\newtheorem{definition}[satz]{Definition}
\newtheorem{bemerkung}[satz]{Remark}

\usepackage{color}
\definecolor{white}{rgb}{1,1,1}
\definecolor{darkred}{rgb}{0.3,0,0}
\definecolor{darkgreen}{rgb}{0,0.3,0}
\definecolor{darkblue}{rgb}{0,0,0.3}
\definecolor{pink}{rgb}{0.78,0.09,0.51}
\definecolor{purple}{rgb}{0.28,0.24,0.55}
\definecolor{orange}{rgb}{1,0.6,0.0}
\definecolor{grey}{rgb}{0.4,0.4,0.4}
\definecolor{aquamarine}{rgb}{0.4,0.8,0.65}

\newcommand{\N}{\mathbb{N}}

\newcommand{\R}{\mathbb{R}}

\newcommand{\norm}[2]{\left\Vert #2 \right\Vert_{L^{#1}(\Omega)}}

\newcommand{\wnorm}[1]{\left\Vert #1 \right\Vert_{W^{1,q}(\Omega)}}
\newcommand{\Norm}[2]{\left\Vert #2 \right\Vert_{#1}}
\newcommand{\tel}[1]{\frac{1}{#1}}
\newcommand{\diff}{\mathop{}\!\mathrm{d}}
\newcommand{\dt}{\frac{\diff}{\diff t}}
\newcommand{\io}{\int_\Omega}

\newcommand{\euler}{\mathrm{e}}

\newcommand{\ubar}[1]{\underaccent{\bar}{#1}}

\newcommand{\eps}{\varepsilon}
\newcommand{\ueps}{u_{\eps}}
\newcommand{\veps}{v_{\eps}}
\newcommand{\ue}{u_{\eps}}
\newcommand{\ve}{v_{\eps}}
\newcommand{\uj}{u_{{\eps}_j}}
\newcommand{\vj}{v_{{\eps}_j}}

\newcommand{\Om}{\Omega}
\newcommand{\Ombar}{\overline{\Omega}}

\newcommand{\f}[2]{\frac{#1}{#2}}
\newcommand{\na}{\nabla}
\newcommand{\nn}{\nonumber}
\newcommand{\kl}[1]{\left(#1\right)}

\usepackage{tikz}
\usepackage{pgfpages}

\usepackage{newunicodechar}
\newunicodechar{ℝ}{\mathbb{R}}
\newunicodechar{ℕ}{\mathbb{N}}
\newunicodechar{Δ}{\Delta}
\newunicodechar{∞}{\infty}
\newunicodechar{ε}{\varepsilon} 
\newunicodechar{λ}{\lambda}
\newunicodechar{φ}{\varphi}
\newunicodechar{τ}{\tau}
\newunicodechar{χ}{\chi}
\newunicodechar{α}{\alpha}
\newunicodechar{β}{\beta}
\newunicodechar{κ}{\kappa}
\newunicodechar{μ}{\mu}

\title{On the global generalized solvability of a chemotaxis model with signal absorption and logistic growth terms}
\author{Elisa Lankeit\\
 {\small Institut f\"ur Mathematik, 
 Universit\"at Paderborn}\\[-0.2cm]
 {\small  Warburger Str. 100}\\[-0.2cm]
 {\small  33098 Paderborn, Germany}\\
{\small elankeit@math.uni-paderborn.de }
 \and
Johannes Lankeit\\
 {\small Institut f\"ur Mathematik, 
 Universit\"at Paderborn}\\[-0.2cm]
 {\small  Warburger Str. 100}\\[-0.2cm]
 {\small  33098 Paderborn, Germany} \\
{\small jlankeit@math.uni-paderborn.de}
}

\begin{document}

\maketitle 
\begin{abstract}
\noindent 
 Introducing a suitable solution concept, we show that in bounded smooth domains $\Omega\subset \mathbb{R}^n$, $n\ge 1$, the initial boundary value problem for the chemotaxis system 
\begin{align*}
 u_t&=\Delta u -\chi\nabla\cdot\left(\frac{u}{v}\nabla v\right)+\kappa u -\mu u^2,\\
 v_t&=\Delta v -uv,
\end{align*}
with homogeneous Neumann boundary conditions and widely arbitrary initial data has a generalized global solution for any $\mu, \kappa, \chi >0$.


\end{abstract}

\section{Introduction}
In the study of chemotaxis systems one of the leading mathematical questions usually is: Does this model admit solutions blowing up (within finite or after infinite time) or are all solutions global and bounded? 
For many systems, the possibility of blow-up is known; for many others, solutions are known to remain bounded (for a multitude of results in both directions consult, e.g., the surveys \cite{BBTW,horstmann}). In between, there still lies a large unchartered territory of models of which it is entirely unknown which of the two mentioned cases they belong to. 

For example, in the systems 
\begin{equation}\label{intro:productionsystem}
 \begin{cases}
  u_t=Δu-χ\nabla\cdot(\f{u}v\na v)\\
  τv_t=Δv - v + u,
 \end{cases}
\end{equation}
describing the prototypical situation of self-aggregating behaviour of cells emitting a signal substance they chemotactically follow in accordance with a singular shape of the sensitivity function (see \cite{kellersegel_trav,kalinin_jiang_tu_wu}), in bounded $n$-dimensional domains and with $τ=1$, it is known that in the case of sufficiently small values of $χ<χ_0(n)$ with $χ_0(2)>1.01$, $χ_0(n)=\sqrt{\f2n}$ for $n\ge 3$ solutions are bounded, \cite{newapproach,Biler99,win-ge,fujie}. 

On the other hand, in the parabolic-elliptic counterparts (with $τ=0$), 
for $χ>\f{2n}{n-2}$ and $n\ge 3$ blow-up can occur, \cite{nagaisenba}. The parabolic-parabolic ($τ=1$) systems \eqref{intro:productionsystem} with large $χ$ belong to the unknown border area previously alluded to. Forays exploring these strange lands have been undertaken in at least two directions: ``Close to'' parabolic--elliptic or elliptic--parabolic settings, that is, for very small or very large values of $τ$,  
Fujie and Senba have demonstrated that boundedness can be achieved, \cite{fujiesenba1,fujiesenba2}; on the other hand, staying with $τ=1$, weaker solution concepts have been pursued in \cite{win-ge,stinnerwinkler,lanwin} and ``weak solutions'', ``weak power-$λ$-solutions'', and ``global generalized solutions'' have been found, when $χ<\sqrt{\f{n+2}{3n-4}}$, $χ<\sqrt{\f n{n-2}}$ and the solutions are radially symmetric, or for $χ<\f{n}{n-2}$ ($n\ge 4$) and $χ<\sqrt{8}$ if $n=3$, respectively. While all of these notions of solutions are compatible with the usual meaning in the sense that if such a solution merely enjoys additional differentiability properties, it already is a classical ($C^{2,1}$-)solution, global existence of solutions in any of these weaker senses does not preclude their unboundedness on some finite time interval. 

Nevertheless, they allow us to gain 
some insight into the possibility of blow-up.  
For example, as long as \[
 χ<\begin{cases}
    ∞,& n=2\\
    \sqrt{8}, &n=3\\
    \f{n}{n-2},&n\ge 4, 
   \end{cases}
\] 
global generalized solutions to \eqref{intro:productionsystem} are obtained in \cite{lanwin} in such a way that, apparently, persistent Dirac-type singularities (those constituting \textit{the} manner of blow-up formation in the classical Keller--Segel system of chemotaxis, 
see \cite{luckhaus-sugiyama-velazquez} 
) are prevented from forming.  

In the system 
\begin{equation}\label{intro:oursystemgeneral}
\begin{cases}  u_t=Δu - χ\nabla\cdot\kl{\f uv \na v} + f(u), \\
  v_t=Δv - uv,\end{cases}
\end{equation}
which we are going to consider in this article, 
and where the cross-diffusive influence in the first equation and evolution of the signal 
interact even more delicately 
than in \eqref{intro:productionsystem}  and even destabilizingly,  
some further indications concerning which conditions lead to global solutions 
would be desirable. 
After all, despite the model (with $f\equiv0$) going back to the 1970s, where it served as macroscopial description for \textit{E. coli} bacteria forming bands, \cite{kellersegel_trav}, and some studies on travelling wave solutions \cite{kellersegel_trav,li_li_wang,nagai_ikeda,TWS-survey}, it was only recently that general existence results were found. 
In \cite{wangxiangyu_asymptotic} a smallness condition on the initial data guaranteeing global existence of bounded solutions in the domains $ℝ^2$ or $ℝ^3$ was discovered. Later it was observed (as a by-product of the analysis in \cite{win_ct_sing_abs_eventual}) that in bounded convex two-dimensional domains another, less restrictive smallness condition has a similar result. Here, moreover, without any smallness requirements, generalized solutions have been found, \cite{win_ct_sing_abs}. In the case of $\io u_0$ being sufficiently small (which is a smallness condition on a biologically interpretable quantity), these have the additional property of eventual regularization, \cite{win_ct_sing_abs_eventual}, so that at least after some unknown but finite time blow-up is impossible. While these results can be recovered if \eqref{intro:oursystemgeneral} is coupled with a fluid, \cite{YilongWang2016,black}, the extension to higher-dimensional settings is not as straightforward. 
In 3D, renormalized solutions have been found, \cite{win_ct_sing_abs_renormalized} -- if the situation is radially symmetric, which is, of course, a rather drastic restriction. 

Several possible changes to the model have been investigated with respect to the question whether they can enable us to find solutions. 
For example, significantly enhancing diffusion at high concentrations (in the form of porous medium type diffusion, that is, by replacing $Δu$ by $Δu^m$ in \eqref{intro:oursystemgeneral}) leads to global solutions and excludes finite-time blowup in bounded domains $\Om\subset ℝ^n$, if $m>1+\f n4$ \cite{locallybounded}. The same effect can be achieved by replacing $\nabla\cdot\kl{\f uv\na v}$ by, essentially, $\nabla\cdot\kl{\f{u^{α}}v\na v}$ with $α<1-\f n4$, \cite{dongmei}. In the two-dimensional setting, also using terms approximately of the form $-u^{β}v$ with $β\in(0,1)$ in place of $-uv$ ensures that the classical solutions exist globally, provided that $χ<1$, \cite{lanvig}, -- and, if $\io u_0$ is small, also entails their boundedness, \cite{lanvig}. 

Another modification of chemotaxis models that can be motivated from biological considerations, and, more importantly, whose presence in many cases serves to supply boundedness of solutions, and which has, hence, been extensively studied (see, e.g. \cite{eventualsmoothness,tian,veryweak,blowupprev,Zhao2016}) 
is that of logistic source terms, i.e. $f(u)=κu-μu^{α}$ ($κ,μ>0$, $α=2$). For $α>1+\f n2$, global classical solutions have been shown to exist in \cite{veryrecent}. As to the case of canonical logistic sources (i.e. $α=2$), in a previous work, \cite{lanlan1}, we have shown that \eqref{intro:oursystemgeneral} with this choice of $f$ has global classical solutions if $0<χ<\sqrt{\f2n}$ and if $μ$ is sufficiently large in the sense that, more precisely, $μ>\f{n-2}n$. Realistic values of $μ$ are positive but small, hence the latter condition is very satisfactory for $n=2$, but in higher dimensions leaves the most interesting cases open. Also, the condition on $χ$ (however much resemblance it bears to the condition needed in the treatment of \eqref{intro:productionsystem} or to that in \cite{lanvig}) raises the question about the remaining parameter range. Accordingly, the main 
question we will pursue in this article is:

\textit{What happens for small values of $μ>0$ (in dimensions $n\ge 3$) -- and what if the assumption $χ<\sqrt{\f2n}$ is removed? }

In line with the above discussion, we will aim for the existence of solutions in a general sense, and hope that the step from $μ=0$ to $μ>0$, in the two-dimensional setting and for small $χ$ responsible for us finding classical instead of generalized solutions, in higher dimensions or for large $χ$ helps us to advance from ``no solutions known at all (apart from a radially symmetric setting in $n=3$)'' to some degree of solvability.

More precisely, we will assume that $\Om\subset ℝ^n$, $n\ge 1$, is a bounded, smooth domain, and that the initial data 
\begin{equation}\label{cond:init}
 u_0\in C^0(\Ombar) \text{ are nonnegative and } \; v_0\in W^{1,\infty}(\Om) \text{ positive throughout $\Ombar$, respectively}.
\end{equation}

We will introduce a concept of generalized solutions (Section \ref{sec:solutionconcept}), and starting from an approximative system with global solutions (see Section \ref{sec:approxsystem}) we will, in Section \ref{sec:absch}, derive estimates allowing us to construct a generalized solution by compactness arguments, so that in Section \ref{sec:constrgensol} we will finally prove the following theorem: 

\begin{satz}\label{thm:main} Let $\Om\subset ℝ^n$, $n\ge 1$, be a bounded, smooth domain, let $u_0$ and $v_0$ satisfy \eqref{cond:init} and let $χ\ge 0$, $κ\ge 0$, $μ>0$ be arbitrary. Then the initial boundary value problem
\begin{align}\label{sys}\begin{array}{rlll}
 u_t&=Δ u -\chi\nabla\cdot\left(\frac{u}{v}\nabla v\right)+\kappa u -\mu u^2,&\text{in }\Omega\times(0,∞),\\
 v_t&=Δ v -uv,&\text{in }\Omega\times(0,∞),\\
 \partial_{\nu}u&=\partial_{\nu}v=0,&\text{in }\partial\Omega\times(0,∞),\\
 u(\cdot,0)&=u_0, \quad v(\cdot,0)=v_0,&\text{in }\Omega,&
 \end{array}
\end{align}
has a global generalized solution in the sense of Definition \ref{def:lsg} below.
\end{satz}

The solution concept we want to consider is based on the notion of solution pursued in  \cite{win_ct_sing_abs}, which in turn is a relative of the concept of renormalized solutions, \cite{boltzmann}. Unlike the system in \cite{win_ct_sing_abs} or other systems without logistic source (as, for example, those in \cite{siam,win_2dCTStokesrot,lanwin}), however, \eqref{sys} does not conserve mass -- a property, on which the solution concepts of the mentioned chemotaxis articles rely heavily. For the definition of subsolutions, we will hence adapt the definition from \cite{veryweak}.

\section{The solution concept}\label{sec:solutionconcept}

We will require the first component $u$ to satisfy two integral inequalities instead of the one integal identity commonly used for the definition of weak solutions. We formalize the first part of the solution concept in the following definition of subsolutions.

\begin{definition}[very weak subsolution]\label{def:sul}
 A pair $(u,v)$ of functions is called \textit{very weak subsolution} to the system \eqref{sys} iff $u$ is nonnegative and $v$ is positive almost everywhere, ${u\in L_{\text{loc}}^2([0,\infty);L^2(\Omega))}$, ${v\in L_{\text{loc}}^2([0,\infty);W^{1,2}(\Omega))}$ and  ${\nabla \log(v)\in L_{\text{loc}}^2(\Omega\times[0,\infty))}$ hold and, moreover, 
  \begin{align}\label{eq:ssul}
  -\int_0^{\infty}\!\!\io \varphi_tu-\io u_0\varphi(\cdot,0)\leq& \int_0^{\infty}\!\!\io uΔ\varphi+\chi\int_0^{\infty}\!\!\io u\nabla\varphi\cdot\nabla\log(v)
  +\kappa\int_0^{\infty}\!\!\io u\varphi-\mu\int_0^{\infty}\!\!\io \varphi u^2
  \end{align} is satisfied for every nonnegative $\varphi\in C_0^{\infty}(\overline{\Omega}\times [0,\infty))$ with $\partial_\nu \varphi=0$ on  $\partial\Omega\times(0,\infty)$ and 
  \begin{align}\label{second:weaksol}
  -\int_0^{\infty}\!\!\io \psi_tv-\io v_0\psi(\cdot,0)=-&\int_0^{\infty}\!\!\io \nabla v \cdot \nabla \psi-\int_0^{\infty}\!\!\io \psi uv
 \end{align}
 is fulfilled for every  $\psi\in C_0^{\infty}(\overline{\Omega}\times [0,\infty))$.
\end{definition}

In addition to this subsolution property, an inequality with the opposite sign will be required for a sensible solution concept. 
\begin{definition}[weak logarithmic supersolution]\label{def:slol}
 A pair of functions $(u,v)$ is called \textit{weak logarithmic supersolution} of (\ref{sys}) iff $u$ is nonnegative and $v$ is positive almost everywhere, $u\in L_{\text{loc}}^1([0,\infty);L^2(\Omega))$, $v\in L_{\text{loc}}^{\infty}(\Om\times[0,\infty))\cap L_{\text{loc}}^2([0,\infty);W^{1,2}(\Omega))$, $\nabla \log(u+1)\in L_{\text{loc}}^2(\Omega\times[0,\infty))$, $\nabla \log(v)\in L_{\text{loc}}^2(\Omega\times[0,\infty))$ and 
 \begin{align}\label{eq:slol}
  -\int_0^{\infty}\!\!\io& \log(u+1)\varphi_t-\io \log(u_0+1)\varphi(\cdot,0) \nn\\
  \geq& -\int_0^{\infty}\!\!\io\nabla\log(u+1)\cdot\nabla\varphi 
  +\int_0^{\infty}\!\!\io\varphi|\nabla\log(u+1)|^2+\chi\int_0^{\infty}\!\!\io\frac{u}{u+1}\nabla \log(v) \cdot \nabla\varphi\nn\\
  &-\chi \int_0^{\infty}\!\!\io \frac{u}{u+1}\varphi \nabla\log(v)\cdot\nabla \log(u+1) 
  + \kappa\int_0^{\infty}\!\!\io\frac{u}{u+1}\varphi-\mu\int_0^{\infty}\!\!\io\frac{u^2}{u+1}\varphi
 \end{align}
is satisfied for every nonnegative $\varphi\in C_0^{\infty}(\overline{\Omega}\times [0,\infty))$ and \eqref{second:weaksol} 
holds for every $\psi\in C_0^{\infty}(\overline{\Omega}\times [0,\infty))$. 
\end{definition}
\begin{bemerkung}
 Because $0\leq \log(u+1)\leq u$ and $\frac{u}{u+1}\leq 1$, all integrals in Definition \ref{def:slol} are well-defined.
\end{bemerkung}

With these two concepts we can now define a generalized solution: 
\begin{definition}[Generalized solution]\label{def:lsg}
 A pair $(u,v)$ of functions is called \textit{generalized solution} to (\ref{sys}), iff $(u,v)$ is a very weak subsolution and a weak logarithmic supersolution to \eqref{sys}.
\end{definition}

This concept of ``generalized solutions'' is compatible with the concept of classical solutions in the following sense:

\begin{satz}
Every pair of functions $(u,v)$ satisfying 
 \begin{align}\label{classsol}
  u&\in C^0(\overline{\Omega}\times[0,∞))\cap C^{2,1}(\overline{\Omega}\times(0,∞))\quad\text{and }\nn\\
 v&\in C^0(\overline{\Omega}\times[0,∞))\cap C^{2,1}(\overline{\Omega}\times(0,∞))\cap L_{\text{loc}}^{\infty}([0,∞);W^{1,2}(\Omega)),
 \end{align}
 and solving \eqref{sys} in the classical sense (hereafter ``classical solution'') is also a generalized solution.
 Also, if $(u,v)$ is a generalized solution to \eqref{sys} which additionally satisfies \eqref{classsol},  then $(u,v)$ is a classical solution of \eqref{sys}.
\end{satz}
\begin{proof}
That every classical solution is a very weak subsolution and a weak logarithmic supersolution follows from testing the PDE by test functions $φ$ or by $φ\cdot\f1{1+u}$, respectively, as soon as the required integrability properties are assured. Concerning the least obvious of these, we note that $v$ is positive by the maximum principle and hence $\nabla \log(v)\in L^2_{loc}(\Ombar\times[0,∞))$ is immediate, and that $\nabla \log(u+1)\in L^2_{loc}(\Ombar\times[0,∞))$ can be obtained from considerations as in Lemma \ref{lm4} below.
Indeed, assuming a sufficient degree of differentiability, like present, the computation in \eqref{auchfuerklassische} can be performed for $ε=0$, too. \\
 We now let $(u,v)$ be a generalized solution to (\ref{sys}) with \eqref{classsol}. Standard arguments relying on the assumed regularity  show that the weak solution property of \eqref{second:weaksol} implies that the second equation of \eqref{sys}, along with its initial and boundary conditions, is also solved classically by $(u,v)$. 
 Since $(u,v)$ is a very weak subsolution, for every nonnegative ${\varphi\in C_0^{\infty}(\overline{\Omega}\times [0,\infty))}$ with $\partial_\nu \varphi=0$ on  $\partial\Omega$ the inequality \eqref{eq:ssul} holds true. Due to \eqref{classsol}, we may integrate by parts and, due to $\nabla\varphi\cdot \nu=0$ and $\nabla \log(v)\cdot \nu=0$ on  $\partial\Omega$, we have 
 \begin{align*}
  \int_0^{\infty}\!\!\io u_t\varphi &+\io u(\cdot,0)\varphi(\cdot,0)-\io u_0\varphi(\cdot,0)\\
  \leq& \int_0^{\infty}\!\!\io Δ u \varphi - \int_0^{\infty} \!\!\int_{\partial\Omega} \nabla u \cdot \nu \varphi-\chi\int_0^{\infty}\!\!\io \nabla \cdot (u\nabla \log(v))\varphi +\kappa\int_0^{\infty}\!\!\io u\varphi-\mu\int_0^{\infty}\!\!\io \varphi u^2.
 \end{align*}
 Here inserting arbitrary smooth nonnegative functions as above, supported in either the interior or close to the spatial or temporal boundary of $\Om\times(0,∞)$, (for a more detailed account of this reasoning see, e.g., the proof of \cite[Lemma 2.5]{lanwin}) we can see that 
 \begin{align}\label{kleinergleich}
  u_t&\leq Δ u -\chi\nabla\cdot(u\nabla\log(v))+\kappa u-\mu u^2&&\text{in }\Omega\times(0,\infty),\nn\\
  \partial_{\nu}u&\leq 0&&\text{on  }\partial\Omega\times(0,∞),\\
  u(\cdot,0)&\leq u_0 &&\text{in }\Omega,\nn 
 \end{align}
 respectively. 
Moreover, $(u,v)$ is a weak logarithmic supersolution, hence for every nonnegative test function $\varphi\in C_0^{\infty}(\overline{\Omega}\times [0,\infty))$ the inequality \eqref{eq:slol} holds.
Integration by parts leads to 
 \begin{align*}
  &\hspace{-0.5cm}\int_0^{\infty}\!\!\!\io \frac{u_t}{u+1}\varphi+\io \log(u(\cdot,0)+1)\varphi(\cdot,0)-\io \log(u_0+1)\varphi(\cdot,0)\\
  \geq &\int_0^{\infty}\!\!\!\!\io \frac{Δ u}{u+1} \varphi -\int_0^{\infty}\!\!\!\!\int_{\partial\Omega} \frac{\varphi}{u+1}\nabla u\cdot \nu -\chi\int_0^{\infty}\!\!\!\!\io\tel{u+1}\nabla\cdot(u\nabla\log(v))\varphi 
  + \kappa\int_0^{\infty}\!\!\!\!\io\frac{u}{u+1}\varphi-\mu\int_0^{\infty}\!\!\!\!\io\frac{u^2}{u+1}\varphi
 \end{align*}
if we use the facts that  $\nabla \log(v)\cdot \nu=0$ on  $\partial\Omega$, and that 
 \begin{align*}
 Δ\log(u+1)&=\frac{Δ u}{u+1}-|\nabla \log(u+1)|^2 \quad\text{ and }\\
 \nabla\cdot \left(\frac{u}{u+1}\nabla \log(v)\right)&=\tel{u+1}\nabla\cdot(u\nabla\log(v))-\frac{u}{u+1}\nabla\log(v)\cdot\nabla\log(u+1).
 \end{align*}
 We can conclude 
 \begin{align*}
  \frac{u_t}{u+1}\geq\frac{Δ u -\chi\nabla\cdot(u\nabla\log(v))+\kappa u-\mu u^2}{u+1},
 \end{align*}
and thus due to nonnegativity of $u$ 
 \begin{align}\label{groessergleich}
  u_t\geq Δ u -\chi\nabla\cdot(u\nabla\log(v))+\kappa u-\mu u^2\qquad \text{in } \Om\times(0,∞).
 \end{align}
Furthermore, as above, we can see that 
  \begin{equation}\label{groessergleichrand}\partial_{\nu}u\geq 0 \text{ on }\partial\Omega\times(0,∞)\end{equation}
 and 
  \(\log(u(\cdot,0)+1)\geq \log(u_0+1)\) in $\Om$,
which due to the monotonicity of $s\mapsto\log(s+1)$ entails
 \begin{align}\label{groessergleichanfang}
  u(\cdot,0)\geq u_0\quad\text{in }\Omega.
 \end{align}
In conclusion, in \eqref{kleinergleich} and \eqref{groessergleich}, \eqref{groessergleichrand}, \eqref{groessergleichanfang}, we have shown that $(u,v)$ satisfies \eqref{sys} classically. 
\end{proof}


\section{An approximating system}\label{sec:approxsystem}
In the following, we will construct a generalized solution as limit of classical solutions to approximating systems. First we will prove global classical solvability of these.

For $\eps>0$ let us consider the system 
\begin{align}\label{syseps}\begin{array}{rlll}
 {\ueps}_t&=Δ {\ueps} -\chi\nabla\cdot\left(\frac{{\ueps}}{(1+\eps\ueps){\veps}}\nabla {\veps}\right)+\kappa {\ueps} -\mu {\ueps}^2,\\
 {\veps}_t&=Δ {\veps} -\frac{\ueps\veps}{(1+\eps\ueps)(1+\eps\veps)},\\
 \partial_{\nu}{\ueps}&=\partial_{\nu}{\veps}=0,\\
 {\ueps}(\cdot,0)&=u_0, ~{\veps}(\cdot,0)=v_0,\\
 \end{array}
\end{align}
with $u_0$, $v_0$ as before. 

Our first goal is to prove the global classical solvability of \eqref{syseps}: 
\begin{lemma}\label{epsglobal}
 Under the assumptions of Theorem \ref{thm:main}, 
 for every $\eps>0$, system (\ref{syseps}) has a global solution. 
\end{lemma}
For the proof we proceed in several steps, the first of which is the local existence of solutions. 
\begin{lemma}\label{epslok}
 Let $\Om\subset ℝ^n$, $n\ge 1$, be a bounded domain with smooth boundary, $\eps>0$ and $q>n$. Then for all nonnegative functions $u_0\in C^0(\overline{\Omega})$ and positive functions $v_0\in W^{1,\infty}(\Omega)$ there are $T_{\max, \eps}\in (0,\infty]$ and a unique pair of functions $(\ueps,\veps)$ satisfying
 \begin{align*}
 \ueps\in C^0(\overline{\Omega}\times[0,T_{\max, \eps}))&\cap C^{2,1}(\overline{\Omega}\times(0,T_{\max, \eps}))\quad\text{and}\\
 \veps\in C^0(\overline{\Omega}\times[0,T_{\max, \eps}))&\cap C^{2,1}(\overline{\Omega}\times(0,T_{\max, \eps}))\cap L_{\text{loc}}^{\infty}([0,T_{\max, \eps});W^{1,q}(\Omega)),
 \end{align*}
 which solves (\ref{syseps}) in the classical sense on  $\Omega\times[0,T_{\max, \eps})$, and for which  \[T_{\max, \eps}=\infty\quad\text{ or }\quad \norm{\infty}{\ueps(\cdot,t)}+\wnorm{\veps(\cdot,t)}\rightarrow \infty \text{ as } t\nearrow T_{\max, \eps}.\]
 The pair $(\ueps,\veps)$ moreover satisfies $\ueps\geq0$, $\veps(\cdot,t)\geq (\inf v_0) \euler^{-\frac{t}{\eps}}>0$ for every $t\in[0,T_{\max, \eps})$.
\end{lemma}
\begin{proof}
 With an analogous approach to that in the proof of \cite[Theorem 2.2]{lanlan1}, local existence and extensibility follow from \cite[Lemma 3.1]{BBTW}; indeed, the situation of \cite[Theorem 2.2]{lanlan1} is more difficult due to the singularities in the system that have been removed in the present setting. Comparison (\cite[Theorem B.1]{lanlan1}) with the subsolution  $\ubar{u}_{\eps}=0$, due to $u_0\geq 0$ shows nonnegativity of $\ueps$. Because of
 \begin{align*}
  \ubar{v}_{\eps t}&=-\tel{\eps}\ubar{v}_{\eps}\leq Δ \ubar{v}_{\eps}-\frac{\ueps \ubar{v}_{\eps}}{(1+\eps\ueps)(1+\eps\ubar{v}_{\eps})} &&\quad\text{in }\Omega\times(0,T_{\max, \eps}),\\
  \ubar{v}_{\eps}(0)&=\inf v_0\leq v_0 &&\quad\text{in }\Omega, \\
  \partial_{\nu}\ubar{v}_{\eps}&=0 &&\quad\text{on  }\partial\Omega,
 \end{align*}
the function $\ubar{v}_{\eps}(x,t):=(\inf v_0) \euler^{-\frac{t}{\eps}}$ is a subsolution and another application of the comparison theorem  also shows  $\veps(\cdot,t)\geq (\inf v_0) \euler^{-\frac{t}{\eps}}>0$ for every $t\in[0,T_{\max, \eps})$.
\end{proof}

From now on, given a domain $\Om$, parameters $χ,κ,μ$ and initial data $(u_0,v_0)$ as in \eqref{cond:init} (in short: under the assumptions of Theorem \ref{thm:main}) we let $(\ueps,\veps)$ denote the unique solution to \eqref{syseps} on  $[0,T_{\max, \eps})$.\\
%

Some simple but important properties of $\ueps$, $\veps$ can be derived immediately and will become essential for the proof of globality of the solutions.
\begin{lemma}\label{lm:veps}
 Under the assumptions of Theorem \ref{thm:main}, for $ε>0$ and $p\in [1,\infty]$ we have \[                                            
\norm{p}{\veps(\cdot,t)}\leq\norm{p}{v_0}\quad\text{ for all }\; t\in [0,T_{\max,\eps}).
                                                                                          \]
\end{lemma}
\begin{proof}
This results from nonnegativity of the derivative $\dt \io \veps^p$ on $(0,T_{\max,\eps})$ for $p\in[1,\infty)$ and, for $p=\infty$, from comparison with the constant supersolution $\norm{\infty}{v_0}$. 
%
\end{proof}

Due to the source terms in \eqref{syseps} being bounded, we can easily derive estimates also for the gradient of $v$ by semigroup estimates -- of course, the size of these bounds will depend on $ε$. 
\begin{lemma}\label{vepsgrad}
 Let $q\in[2,\infty)$. For every finite $T\leq T_{\max, \eps}$ there is $C=C(\eps,T)>0$, so that $\norm{q}{\nabla \veps(\cdot,t)}\leq C$ for  $t\in(0,T)$. 
\end{lemma}
\begin{proof}
 According to Duhamel's formula, we can represent $\nabla \veps$ as 
 \begin{align*}
  \nabla \veps(\cdot,t)=\nabla \left(\euler^{t Δ}v_0\right)+\int_0^t\nabla \left(\euler^{(t-s)Δ}\frac{\ueps\veps}{(1+\eps\ueps)(1+\eps\veps)}\right)\diff s, \quad t\in(0,T_{\max,\eps}).
 \end{align*}
 Using the obvious estimate $\frac{a}{1+\eps a}\leq \tel{\eps}$ for $a\geq 0$ and semigroup estimates, we see that for every $t\in(0,T)$ we have 
 \begin{align*}
  \norm{q}{\nabla \veps(\cdot,t)}&\leq \norm{q}{\nabla \left(\euler^{t Δ}v_0\right)}+\int_0^t\norm{q}{\nabla \left(\euler^{(t-s)Δ}\frac{\ueps\veps}{(1+\eps\ueps)(1+\eps\veps)}\right)}\diff s\\
  &\leq c_1\norm{q}{\nabla v_0}+c_1\int_0^t\left(1+(t-s)^{-\tel{2}}\right)\norm{q}{\frac{\ueps}{(1+\eps\ueps)}\f{\veps}{(1+\eps\veps)}}\diff s\\
  &\leq c_1\norm{q}{\nabla v_0}+\frac{c_1|\Omega|^{\tel{q}}}{\eps^2}\int_0^T\left(1+(t-s)^{-\tel{2}}\right)\diff s=:C(\eps,T),
 \end{align*}
 where $c_1$ is the constant obtained from the semigroup estimates of \cite[Lemma 1.3 iii)]{lplq}.
\end{proof}
This ensures that also $\ueps$ remains bounded. 
\begin{lemma}\label{lm:ueps}
 Under the assumptions of Theorem \ref{thm:main}, for every $ε>0$ and every $T\leq T_{\max, \eps}$ with $T<\infty$ there is $C=C(\eps, T)>0$, so that $\norm{\infty}{\ueps(\cdot,t)}\leq C$ on  $(0,T)$.
\end{lemma}
\begin{proof}
 The logistic map $f\colon \R\rightarrow \R, f(s)=\kappa s-\mu s^2$, satisfies $f(s)\leq \frac{\kappa^2}{4\mu}$ for every $s\in\R$.\\
 Let $\hat{u}_{\eps}$ be the solution to 
 \begin{align*}
  \hat{u}_{\eps t}=Δ\hat{u}_{\eps}-\chi\nabla\cdot\left(\frac{\hat{u}_{\eps}}{(1+\eps\hat{u}_{\eps}){\veps}}\nabla {\veps}\right)+\frac{\kappa^2}{4\mu},\quad \text{in } \Omega\times (0,T_{\max, \eps}).
 \end{align*}
 Then, due to the comparison theorem \cite[Thm. B.1]{lanlan1}, we have $\ueps\leq \hat{u}_{\eps}$ and hence, according to $\ueps\geq0$, also $\norm{\infty}{\ueps(\cdot,t)}\leq\norm{\infty}{\hat{u}_{\eps}(\cdot,t)}$.\\
 Now we consider $\norm{\infty}{\hat{u}_{\eps}(\cdot,t)}$. By representation of $\hat{u}$ in terms of the semigroup, corresponding  estimates (see \cite[Lemma 1.3 iv)]{lplq}) then for $t\in(0,T)$ show that with some $c_1>0$
 \begin{align*}
  \norm{\infty}{\hat{u}_{\eps}(\cdot,t)}\leq& \norm{\infty}{\euler^{tΔ}u_0}+\frac{T\kappa^2}{4\mu} 
  +\chi\int_0^t\norm{\infty}{\euler^{(t-s)Δ}\nabla\cdot\left(\frac{\hat{u}_{\eps}}{(1+\eps\hat{u}_{\eps}){\veps}}\nabla {\veps}\right)}\diff s\\
  \leq& \norm{\infty}{u_0}+\frac{T\kappa^2}{4\mu}
  +\chi c_1\int_0^t\left(1+(t-s)^{-\tel{2}-\frac{n}{2n+2}}\right)\norm{n+1}{\frac{\hat{u}_\eps}{1+\eps\hat{u}_\eps}\frac{\nabla \veps}{\veps}}\diff s\\
  \leq& \norm{\infty}{u_0}+\frac{T\kappa^2}{4\mu}
  +\frac{\chi c_1}{\eps}\int_0^t\left(1+(t-s)^{-\tel{2}-\frac{n}{2n+2}}\right)\norm{2n+2}{\nabla \veps}\norm{2n+2}{\tel{\veps}}\diff s.
 \end{align*}
 According to Lemma \ref{vepsgrad} there is $c_2>0$ satisfying $\norm{2n+2}{\nabla \veps(\cdot,t)}\leq c_2$ for $t\in(0,T)$. Moreover, 
 \begin{align*}
  \norm{2n+2}{\tel{\veps(\cdot,t)}}\leq \frac{\euler^{\eps T}}{\inf v_0}|\Omega|^{\tel{2n+2}}\qquad\text{for } t\in(0,T).
 \end{align*}
 Because, due to $-\tel{2}-\frac{n}{2n+2}=-\frac{2n+1}{2n+2}>-1$, the remaining integral 
 \begin{align*}
 \int_0^t\left(1+(t-s)^{-\tel{2}-\frac{n}{2n+2}}\right)\diff s
 \end{align*}
 is also finite and bounded independently of $t\in(0,T)$, we can conclude the existence of $C(ε,T)>0$ such that $\norm{\infty}{\hat{u}_{\eps}(\cdot,t)}\leq C(ε,T)$ on  $(0,T)$. Together with $\norm{\infty}{\ueps(\cdot,t)}\leq\norm{\infty}{\hat{u}_{\eps}(\cdot,t)}$, the claim follows.
\end{proof}
Summarily, these results show that $T_{\max, \eps}=\infty$ for every $\eps>0$:
\begin{proof}[Proof of Lemma \ref{epsglobal}]
 Suppose, $T_{\max, \eps}<\infty$. By Lemma \ref{epslok}, then 
\begin{align*}\norm{\infty}{\ueps(\cdot,t)}+\wnorm{\veps(\cdot,t)}\rightarrow \infty\qquad \text{ as } t\nearrow T_{\max, \eps}\end{align*}
 would have to hold. Lemmata \ref{lm:veps}, \ref{vepsgrad} and \ref{lm:ueps}, however, exclude this possibility.
\end{proof}
\section{A priori estimates}\label{sec:absch}
In order to obtain generalized solutions to \eqref{sys} from classical solutions of \eqref{syseps}, we will now derive $ε$-independent estimates for $\ueps$, $\veps$, $\log(\veps)$ and $\log(\ueps+1)$ in suitable spaces. The constants $C_i$ arising therein will continue to be used in the subsequent lemmata.

We begin with boundedness of the total bacterial mass, which is easily obtained even though the logistic source removes the mass conservation property many chemotaxis systems have. 
\begin{lemma}\label{lm1}
Under the assumptions of Theorem \ref{thm:main}, there is $C_1>0$ such that for all $\eps>0$
 \begin{align*}
  \norm{1}{\ueps(\cdot,t)}\leq C_1\quad\text{for all } t\in [0,\infty).
 \end{align*}
\end{lemma}
\begin{proof}
Due to $\ueps\geq 0$ we have $\io \ueps(\cdot,t)=\norm{1}{\ueps(\cdot,t)}$. We can derive a differential inequality for $\io \ueps$: 
 \begin{align*}
  \dt \io \ueps =\io u_{\eps t} = \kappa \io \ueps -\mu \io \ueps^2\leq \kappa \io \ueps -\frac{\mu}{|\Omega|}\left(\io \ueps\right)^2 \qquad \text{on }(0,\infty),
 \end{align*}
and the claim follows by an ODI comparison immediately.
\end{proof}

The spatio-temporal $L^2$ estimate we are about to obtain in the following lemma heavily relies on presence of the logistic source. 
\begin{lemma}\label{lm2}
Under the assumptions of Theorem \ref{thm:main}, for every $T>0$ there is $C_2=C_2(T)>0$, such that for all $\eps>0$ we have
 \begin{align*}
  \int_0^T\!\!\io \ueps^2\leq C_2.
 \end{align*}
\end{lemma}
\begin{proof}
We isolate $\ueps^2$ in the first equation of (\ref{syseps}), integrate over $\Omega\times[0,T)$ and use Lemma \ref{lm1}: 
 \begin{equation*}
  \int_0^T\!\!\io \ueps^2=\tel{\mu}\left(\kappa\int_0^T\!\!\io\ueps-\io \ueps(\cdot,T)+\io u_0 \right) 
  \leq \frac{(\kappa T+1)C_1}{\mu}=:C_2(T)\quad \text{for any } \eps>0.\qedhere
 \end{equation*}
\end{proof}

On our way to further estimates on $\ue$, especially concerning its derivatives, we include some gradient information for $\ve$ by means of its logarithm -- which is how $\ve$ appears in the chemotaxis term. In Section \ref{sec:nablalogv} we will deal with $\nabla \log \ve$ in some more detail, since its convergence will play a crucial role in finding the generalized solution we are searching for. 
\begin{lemma}\label{lm3}
Under the assumptions of Theorem \ref{thm:main}, for every $T>0$ there is $C_3=C_3(T)>0$ such that for all $\eps>0$
 \begin{align*}
  \int_0^T \!\!\io |\nabla \log(\veps)|^2\leq C_3.
 \end{align*}
\end{lemma}
\begin{proof} We fix $T>0$. Denoting $w_{\eps}:=-\log\left(\frac{\veps}{\norm{\infty}{v_0}}\right)$, we see that in $\Om\times(0,T)$ 
 \begin{align*}
  \nabla w_{\eps}&=-\nabla \log\left(\frac{\veps}{\norm{\infty}{v_0}}\right)=-\nabla\log(\veps)+\nabla\log(\norm{\infty}{v_0})=-\nabla \log(\veps),
   \end{align*}
as well as 
 \begin{align*}
  w_{\eps}&\geq 0\quad\text{and }\quad w_{\eps t}=Δ w_{\eps}-|\nabla w_{\eps}|^2+\frac{\ueps}{(1+\eps\ueps)(1+\eps\veps)}
 \end{align*}
hold for any $\eps>0$. Hence on $(0,T)$ 
 \begin{align*}
  \io |\nabla \log(\veps)|^2=\io Δ w_{\eps} +\io \underbrace{\frac{\ueps}{(1+\eps\ueps)(1+\eps\veps)}}_{\leq \ueps} -\io w_{\eps t}\leq C_1 -\io w_{\eps t} \quad \text{ for any }\eps>0 
 \end{align*}
by Lemma \ref{lm1}. Integration with respect to time due to nonnegativity of $w_{\eps}$ results in 
 \begin{equation*}
  \int_0^T \!\!\io |\nabla \log(\veps)|^2 \leq C_1 T-\io w_{\eps}(\cdot,T)+\io w_0 
  \leq C_1 T-\io \log\left(\frac{v_0}{\norm{\infty}{v_0}}\right)=:C_3(T) \; \text{ for any }\eps>0.\qedhere
 \end{equation*}
\end{proof}

As to derivative information on $\ue$, it is possible to garner the results of some differential inequality satisfied by the integral of its logarithm. This proof already employs the result of Lemma \ref{lm3}. 

\begin{lemma}\label{lm4}
Under the assumptions of Theorem \ref{thm:main}, for every $T>0$ there is $C_4=C_4(T)>0$, such that 
 \begin{align*}
  \int_0^T\!\!\io |\nabla \log(\ueps+1)|^2\leq C_4
 \end{align*}
 for all $\eps>0$. 
\end{lemma}
\begin{proof}
 We let $T>0$, fix $C_1$ and $C_3$ as in Lemma \ref{lm1} and Lemma \ref{lm3}, respectively, and let $\eps>0$. 
 Because $\ueps\ge 0$, also $\log(\ueps+1)\geq 0$. Moreover, for every $s>0$, apparently $\log(s+1)\leq s$. 
 Computing the time derivative of $-\io\log(\ueps+1)$, taking into account $\nabla \tel{\ueps+1}=-\frac{\nabla\ueps}{(\ueps+1)^2}$, from integration by parts we infer 
 \begin{align*}
  \dt&\left(-\io\log(\ueps+1)\right)\\
  &=-\io |\nabla \log(\ueps+1)|^2+\chi\io \frac{\ueps}{(\ueps+1)(1+\eps\ueps)}\nabla \log(\ueps+1) \cdot \nabla \log{\veps}
  -\kappa\io\frac{\ueps}{\ueps+1}+\mu\io\frac{\ueps^2}{\ueps+1}.
 \end{align*}
 Integration over $(0,T)$ with aid of Lemmata \ref{lm1} and \ref{lm3} and Young's inequality shows 
 \begin{align}\label{auchfuerklassische}
  \int_0^T\!\!\io &|\nabla\log(\ueps+1)|^2=\io \underbrace{\log(\ueps(\cdot,T)+1)}_{\leq \ueps(\cdot,T)}\underbrace{-\io \log(u_0+1)}_{\leq0} \nn\\
  &+\chi\int_0^T\!\!\io \underbrace{\frac{\ueps}{(\ueps+1)(1+\eps\ueps)}}_{\leq 1}\nabla \log(\ueps+1) \cdot \nabla \log{\veps} 
  \underbrace{-\kappa\int_0^T\!\!\io\frac{\ueps}{\ueps+1}}_{\leq 0}+\mu\int_0^T\!\!\io\underbrace{\frac{\ueps^2}{\ueps+1}}_{\leq \ueps}\nn\\
  \leq& C_1+\tel{2}\int_0^T\!\!\io |\nabla\log(\ueps+1)|^2 + \frac{\chi^2}{2}\int_0^T\!\!\io |\nabla\log(\veps)|^2+\mu C_1 T\nn\\
  \leq& (1+\mu T)C_1+\frac{\chi^2}{2}C_3+\tel{2}\int_0^T\!\!\io |\nabla\log(\ueps+1)|^2,
 \end{align}
 and hence 
 \begin{equation*}
  \int_0^T\!\!\io |\nabla\log(\ueps+1)|^2\leq 2(1+\mu T)C_1+\chi^2 C_3=:C_4.\qedhere
 \end{equation*}
\end{proof}

This bound on the logarithm actually entails an estimate for $\na \ue$ itself, thanks to the bound on $\int_0^T\io \ue^2$ obtained earlier. 
\begin{lemma}\label{lm5}
 Under the assumptions of Theorem \ref{thm:main}, for every $T>0$ there is $C_5=C_5(T)$, such that 
 \begin{align*}
  \Norm{L^1(\Omega\times(0,T))}{\nabla \ueps}\leq C_5 
 \end{align*}
 for all $\eps>0$.
\end{lemma}
\begin{proof} We let $T>0$, fix $C_1$, $C_2$, $C_4$ as before and let $\eps>0$. 
Young's inequality together with Lemmata \ref{lm4}, \ref{lm1} and \ref{lm2} shows 
\begin{align*}
  \Norm{L^1(\Omega\times(0,T))}{\nabla \ueps}&=\int_0^T\!\!\io |\nabla\ueps|=\int_0^T\!\!\io \frac{|\nabla \ueps|}{\ueps+1}(\ueps+1)\\
  &\leq \tel{2}\int_0^T\!\!\io \frac{|\nabla\ueps|^2}{(\ueps+1)^2}+\tel{2}\int_0^T\!\!\io (\ueps+1)^2\\
  &=\tel{2}\underbrace{\int_0^T\!\!\io |\nabla \log(\ueps+1)|^2}_{\leq C_4}+\tel{2}\underbrace{\int_0^T\!\!\io \ueps^2}_{\leq C_2}+\int_0^T\underbrace{\io \ueps}_{\leq C_1}+\tel{2}|\Omega|T\\
  &\leq \frac{C_4}{2}+\frac{C_2}{2}+C_1T+\frac{|\Omega|T}{2}=:C_5.\qedhere
\end{align*}
\end{proof}

An immediate consequence of Lemma \ref{lm1} and Lemma \ref{lm5} is the following: 

\begin{kor}\label{kor6}
 Under the assumptions of Theorem \ref{thm:main}, for every $T>0$ there is $C_6=C_6(T)$, such that
 \begin{align*}
  \Norm{L^1((0,T);W^{1,1}(\Omega))}{\ueps}\leq C_6
 \end{align*}
 for all $\eps>0$. 
\end{kor}

Estimates for $\ue$ ensured, we now turn our attention to $\na\ve$.  

\begin{lemma}\label{lm7}
 Under the assumptions of Theorem \ref{thm:main}, there is $C_7>0$, such that for all $\eps>0$ 
 \begin{align*}
  \norm{2}{\nabla \veps(\cdot,t)}\leq C_7
 \end{align*}
 holds for $t\in[0,\infty)$.\\
Moreover, for every $T>0$ there is $C_8=C_8(T)>0$, such that
 \begin{align*}
  \Norm{L^2(\Omega\times(0,T))}{Δ \veps}\leq C_8
 \end{align*}
 holds for all $\eps>0$.
\end{lemma}
\begin{proof}
We will derive a differential inequality for $y_\eps(t):=\io |\nabla \veps(\cdot,t)|^2$. 
On $(0,\infty)$ we have 
 \begin{align*}
  \dt \io |\nabla \veps|^2&=2\io \nabla \veps\cdot\nabla v_{\eps t}
  =2\io \nabla \veps\cdot\nabla Δ \veps-2\io \nabla \veps\cdot\nabla \left(\frac{\ueps\veps}{(1+\eps\ueps)(1+\eps\veps)}\right).
 \end{align*}
With integration by parts and Young's inequality we have 
 \begin{align}\label{eq:gradveps}
  \dt \io |\nabla \veps|^2&=-2\io |Δ \veps|^2+2\io Δ \veps\underbrace{\frac{\ueps\veps}{(1+\eps\ueps)(1+\eps\veps)}}_{\leq \ueps\veps}\nonumber\\
  &\leq -2 \io |Δ \veps|^2+\io |Δ \veps|^2+\io \ueps^2\veps^2\\
  &\leq -\io |Δ \veps|^2 +\norm{\infty}{v_0}^2\io \ueps^2\text{ on } (0,\infty)\nonumber,
 \end{align}
where in the last step we used $\veps^2\leq \norm{\infty}{v_0}^2$ in accordance with Lemma \ref{lm:veps}. Poincar\'{e}'s inequality furthermore yields $C_P>0$ satisfying 
 \begin{align*}
 -\io |Δ \veps(\cdot,t)|^2\leq -\tel{C_P}\io |\nabla \veps(\cdot,t)|^2 \qquad \text{for all } t>0,\; ε>0, 
 \end{align*}
so that, in conclusion, we obtain the differential inequality
 \begin{align*}
  \dt \io |\nabla \veps|^2+ \tel{C_P}\io |\nabla \veps|^2\leq \norm{\infty}{v_0}^2\io \ueps^2\quad \text{ on } (0,\infty)\quad\text{ for any } ε>0, 
 \end{align*}
where the integral in time of the right hand side can be controlled by Lemma \ref{lm2}. We can hence conclude 
the existence of a constant $\tilde{C_7}$ with $\io |\nabla \veps|^2\leq \tilde{C_7}$ on  $(0,\infty)$ (for the elementary proof see, e.g., \cite[Lemma 3.4]{lanlan1}). The first assertion follows upon letting $C_7:=\sqrt{\tilde{C_7}}$.\\
With the boundedness of $\io |\na \ve|^2$ guaranteed, differential inequality \eqref{eq:gradveps} actually has a second useful consequence:  
 From (\ref{eq:gradveps}) by integration from $0$ to $T$:
 \begin{align*}
   \Norm{L^2(\Omega\times(0,T))}{Δ \veps}^2&=\int_0^T\!\!\io |Δ \veps|^2\leq \norm{\infty}{v_0}^2\int_0^T\!\!\io \ueps^2+\io |\nabla v_0|^2-\io |\nabla \veps(\cdot, T)|^2\\
   &\leq \norm{\infty}{v_0}^2C_2+\io |\nabla v_0|^2=:\tilde{C_8} \quad \text{ for all }ε>0.
 \end{align*}
 Defining $C_8:=\sqrt{\tilde{C_8}}$ concludes the proof.
\end{proof}

At this point, we can control all terms on the right-hand side of the second equation of \eqref{syseps}, and accordingly also $v_{εt}$: 
\begin{kor}\label{kor9}
Under the assumptions of Theorem \ref{thm:main}, for every $T>0$ there is $C_9=C_9(T)>0$, such that for all $\eps>0$
 \begin{align*}
  \Norm{L^2(\Omega\times(0,T))}{v_{\eps t}}\leq C_9.
 \end{align*}
\end{kor}
\begin{proof}
 We fix $T>0$. From the second part of Lemma \ref{lm7}, Lemma \ref{lm2} and Lemma \ref{lm:veps} we obtain that for all $ε>0$ 
 \begin{align*}
  \Norm{L^2(\Omega\times(0,T))}{v_{\eps t}}&\leq \Norm{L^2(\Omega\times(0,T))}{Δ \veps}+\Norm{L^2(\Omega\times(0,T))}{\frac{\ueps\veps}{(1+\eps\ueps)(1+\eps\veps)}}\\
  &\leq \Norm{L^2(\Omega\times(0,T))}{Δ \veps} + \norm{\infty}{v_0}\Norm{L^2(\Omega\times(0,T))}{\ueps}\\
  &\leq C_8+\norm{\infty}{v_0}\sqrt{C_2}=:C_9.\qedhere 
 \end{align*}
\end{proof}

Previously, we have obtained estimates for $\na \log\ve$. We now complement this by a bound for $\log \ve$ itself. 
\begin{lemma}\label{lm10}
Under the assumptions of Theorem \ref{thm:main}, for every $T>0$ there is $C_{10}=C_{10}(T)>0$, such that for all $\eps>0$
 \begin{align*}
  \norm{1}{\log(\veps(\cdot,t))}\leq C_{10}\qquad \text{for all } t\in (0,T).
 \end{align*}
\end{lemma}
\begin{proof}
 Due to $\nabla \tel{\veps}=-\frac{\nabla \veps}{\veps^2}$, integration by parts yields 
 \begin{align*}
  \dt \io (-\log(\veps))&=-\io \frac{v_{\eps t}}{\veps} = -\io \frac{Δ \veps}{\veps}+\io \frac{\ueps}{(1+\eps\ueps)(1+\eps\veps)} \\
  &=\io \nabla \veps \cdot \nabla\left(\tel{\veps}\right)+\io \frac{\ueps}{(1+\eps\ueps)(1+\eps\veps)} \\
  &\leq -\io |\nabla\log(\veps)|^2+\io \ueps\leq \io \ueps.
 \end{align*}
 Integrating this between $0$ and $T$, taking into account Lemma \ref{lm1} we obtain 
 \begin{align*}
  \io -\log(\veps)\leq -\io \log(v_0)+\int_0^T\!\!\io \ueps \leq -\io \log(v_0)+C_1T.
 \end{align*}
 Because $|\log(s)|\leq 2s-\log(s)$ holds for all $s>0$, in combination with Lemma \ref{lm:veps} we may conclude that 
 \begin{equation*}
  \norm{1}{\log(\veps)}=\io |\log(\veps)|\leq 2\io \veps - \io \log(\veps)
  \leq 2 \norm{1}{v_0}-\io \log(v_0)+C_1T=:C_{10}.\qedhere
 \end{equation*}
\end{proof}


In conclusion, this means the following for $\log \ve$: 
\begin{lemma}\label{lm11}
 Under the assumptions of Theorem \ref{thm:main}, for every $T>0$ there is $C_{11}=C_{11}(T)>0$, such that for all $\eps>0$
 \begin{align*}
  \Norm{L^2((0,T);W^{1,2}(\Omega))}{\log(\veps)}\leq C_{11}.
 \end{align*}
\end{lemma}
\begin{proof}
Poincar\'e's inequality provides us with a constant $C_P>0$ such that 
\[
 \io y^2\le C_P \io |\na y|^2 + |\Om|^{-1}\norm1{y}^2\qquad \text{for all } y\in W^{1,2}(\Om).
\]
Taken together with Lemma \ref{lm3} and Lemma \ref{lm10}, this entails 
  \begin{align*}
  &\Norm{L^2((0,T);W^{1,2}(\Omega))}{\log(\veps)}^2=\int_0^T\!\!\io |\log(\veps)|^2+\int_0^T\!\!\io |\nabla\log(\veps)|^2\\
  &\leq C_P\int_0^T\!\!\io |\nabla \log(\veps)|^2+|\Omega|^{-1}\int_0^T\norm{1}{\log(\veps)}^2+\int_0^T\!\!\io |\nabla\log(\veps)|^2\\
  &\leq (C_P+1)C_3 +|\Omega|^{-1}TC_{10}^2=:\tilde{C}_{11}.
  \end{align*}
The definition $C_{11}:=\sqrt{\tilde{C}_{11}}$ directly results in the above claim. 
\end{proof}

In order to apply the Aubin--Lions lemma, we will additionally require estimates for the time derivatives of  $\ueps$ and $\log(\veps)$. 
For $v_{\eps t}$ we have already obtained a bound in $L^2(\Omega\times(0,T))$ in Corollary \ref{kor9}. While we cannot expect to find estimates for $u_{\eps t}$ and $\left(\log(\veps)\right)_t$ in such a ``good'' space, the following lesser regularity assertions will be sufficient: 

\begin{lemma}\label{lm_ut}
Under the assumptions of Theorem \ref{thm:main}, for every $T>0$, the set $\left\{u_{\eps t}\right\}_{\eps\in(0,1)}$ is bounded in $L^1\left((0,T);(W_0^{2,\infty}(\Omega))^*\right)$. 
\end{lemma}
\begin{proof}
 By definition of the norm and density of $C_0^\infty(\Om)$ in $W_0^{2,\infty}(\Omega)$, we have that 
 \begin{align*}
   &\Norm{L^1\left((0,T);(W_0^{2,\infty}(\Omega))^*\right)}{u_{\eps t}}=\int_0^T\sup_{\substack{\varphi \in W_0^{2,\infty}(\Omega)\\ \Norm{W_0^{2,\infty}}{\varphi}\leq 1}} \left\lvert\io u_{\eps t}\varphi\right\rvert=\int_0^T\sup_{\substack{\varphi \in C_0^{\infty}(\Omega)\\ \Norm{W_0^{2,\infty}}{\varphi}\leq 1}} \left\lvert\io u_{\eps t}\varphi\right\rvert,
\end{align*}
and for any $φ\in C_0^\infty(\Om)$ satisfying $\Norm{W^{2,∞}_0(\Om)}{φ}\le 1$, the equation for $u_{\eps t}$ and integration by parts show 
\begin{align*}\int_0^T \left\lvert\io u_{\eps t}\varphi\right\rvert 
   &\leq \int_0^T \left\lvert\io Δ\ueps\varphi\right\rvert+\int_0^T \chi\left\lvert\io\nabla\cdot \left(\frac{\ueps}{1+\eps\ueps}\frac{\nabla\veps}{\veps}\right)\varphi\right\rvert
   +\int_0^T \kappa\left\lvert\io\ueps\varphi\right\rvert+\int_0^T \mu\left\lvert\io\ueps^2\varphi\right\rvert\\
   &\leq \int_0^T \left\lvert\io \uepsΔ\varphi\right\rvert+\int_0^T \chi\left\lvert\io \frac{\ueps}{1+\eps\ueps}\frac{\nabla\veps}{\veps}\cdot\nabla \varphi\right\rvert
   +\int_0^T \kappa\left\lvert\io\ueps\varphi\right\rvert+\int_0^T \mu\left\lvert\io\ueps^2\varphi\right\rvert\\
   &\leq\int_0^T \norm{\infty}{Δ\varphi}\left\lvert\io \ueps\right\rvert+\int_0^T\norm{\infty}{\nabla\varphi} \chi\left\lvert\io \ueps \nabla\log(\veps)\right\rvert\\
   &\hspace{1.5cm}+\int_0^T \kappa\norm{\infty}{\varphi}\left\lvert\io\ueps\right\rvert+\int_0^T \mu\norm{\infty}{\varphi}\left\lvert\io\ueps^2\right\rvert\\
   &\leq \int_0^T\!\!\io \ueps + \frac{\chi}{2} \int_0^T\!\!\io \ueps^2 +\frac{\chi}{2}\int_0^T\!\!\io |\nabla\log(\veps)|^2 +\kappa\int_0^T\!\!\io \ueps +\mu\int_0^T\!\!\io\ueps^2\\
   &\leq (1+\kappa)C_1T+\left(\frac{\chi}{2}+\mu\right)C_2+\frac{\chi}{2}C_3
 \end{align*}
if we apply Young's inequality and Lemmata \ref{lm1}, \ref{lm2} and \ref{lm3}. 
\end{proof}

\begin{lemma}\label{lm_logvt}
Under the assumptions of Theorem \ref{thm:main}, for every $T>0$, the set $\left\{(\log(\veps))_t\right\}_{\eps\in(0,1)}$ is bounded in $L^1\left((0,T);(W_0^{2,\infty}(\Omega))^*\right)$.
\end{lemma}
\begin{proof}
For any $φ\in C_0^\infty(\Om)$ satisfying $\Norm{W^{2,∞}_0(\Om)}{φ}\le 1$, the equation for $v_{\eps t}$, integration by parts and Young's inequality result in 
  \begin{align*}
\int_0^T&\left\lvert\io (\log(\veps))_t\varphi\right\rvert 
    =\int_0^T \left\lvert\io (\log(\veps))_t\varphi\right\rvert\\
    =& \int_0^T \left\lvert\io \frac{Δ \veps}{\veps}\varphi-\io\frac{\ueps}{(1+\eps\ueps)(1+\eps\veps)}\varphi\right\rvert\\
    \leq& \int_0^T \left\lvert\io \frac{|\nabla \veps|^2}{\veps^2}\varphi
    -\io \frac{\nabla\veps}{\veps}\cdot\nabla\varphi-\io\frac{\ueps}{(1+\eps\ueps)(1+\eps\veps)}\varphi\right\rvert\\
    \leq& \int_0^T \norm{\infty}{\varphi}\io|\nabla\log(\veps)|^2 
     + \int_0^T \left(\tel{2}\io |\nabla\log(\veps)|^2+\tel{2}\io |\nabla\varphi|^2\right)
     +\int_0^T \norm{\infty}{\varphi}\io\ueps \\
     \leq& \frac{3C_3}{2}+\left(\frac{|\Omega|}{2}+C_1\right)T,
  \end{align*}
 where the last step relies on Lemmata \ref{lm1} and \ref{lm3}, so that 
 definition of the norm and density of $C_0^\infty(\Om)$ in $W_0^{2,\infty}(\Om)$ show that 
\[
 \Norm{L^1\left((0,T);(W_0^{2,\infty}(\Omega))^*\right)}{(\log(\veps))_t}\leq \frac{3C_3}{2}+\left(\frac{|\Omega|}{2}+C_1\right)T.\qedhere
\]
\end{proof}

\section{Construction of a generalized solution}\label{sec:constrgensol}
Aided by the estimates from Section \ref{sec:absch}, the Aubin--Lions lemma and some further (basic) functional analytic properties, we now construct a generalized solution $(u,v)$ as limit of a subsequence of $(\ueps,\veps)$. 


Firstly, we ensure convergence of the first component in a pointwise sense. 
\begin{lemma}\label{aubinu}
 Under the assumptions of Theorem \ref{thm:main}, there are $u\in L^1_{loc}(\Ombar\times[0,∞))$ and a sequence $\eps_j\searrow 0$ such that $u_{\eps_j}\rightarrow u$ in $L^1(\Omega\times(0,T))$ for any $T>0$ and $u_{\eps_j}\rightarrow u$ almost everywhere in $\Omega\times(0,∞)$ as $j\to \infty$.
\end{lemma}
\begin{proof}
 We have that 
 \( W^{1,1}(\Omega)\hookrightarrow L^1(\Omega)\hookrightarrow (W_0^{2,\infty}(\Omega))^*\),
 where the embedding $W^{1,1}(\Omega)\hookrightarrow L^1(\Omega)$ is compact. By the Aubin--Lions lemma (\cite[Cor. 8.4]{Simon1986}), for any $T>0$, Corollary \ref{kor6} and Lemma \ref{lm_ut} entail relative compactness of $\left\{\ueps\right\}_{\eps\in(0,1)}$ in $L^1((0,T);L^1(\Omega))=L^1(\Omega\times(0,T))$. 
 Consequently, there is a sequence $\eps_j\searrow 0$ such that  $u_{\eps_j}\rightarrow u$ in $L^1(\Omega\times(0,T))$, and a.e. convergence along a subsequence results as well. 
 A diagonalization procedure (for more explicit details consult \cite[Section 4]{locallybounded}) ensures existence of $u$ on $\Om\times(0,∞)$ and independence of $(ε_j)_j$ from the choice of $T$. 
\end{proof}

Similarly, a limit of the second components of the solutions can be obtained. 
\begin{lemma}\label{aubinv}
 Under the assumptions of Theorem \ref{thm:main}, there is $v\in L^2_{loc}(\Ombar\times[0,∞))$ such that (along a subsequence of $(ε_j)_{j\inℕ}$ from Lemma \ref{aubinu}) $v_{\eps_j}\rightarrow v$ in $L^2(\Omega\times(0,T))$ and almost everywhere, as $j\to \infty$. 
 In particular, $v_{\eps_j}(\cdot,t)\to v(\cdot,t)$ in $L^2(\Om)$ as $j\to \infty$, for almost every $t>0$.
\end{lemma}
\begin{proof} For $T>0$, 
 Lemma \ref{lm7} and Lemma \ref{lm:veps} imply that $\left\{\veps\right\}_{\eps\in(0,1)}$ is bounded in $L^{\infty}\left((0,T);(W^{1,2}(\Omega)\right)$ and hence also in $L^2((0,T);W^{1,2}(\Omega))$. Furthermore, according to Corollary \ref{kor9}, $\left\{v_{\eps t}\right\}_{\eps\in(0,1)}$ is bounded in $L^2((0,T);L^2(\Omega))$. Moreover we have 
 \( W^{1,2}(\Omega)\hookrightarrow L^2(\Omega) \hookrightarrow L^2(\Omega)\)
 where  $W^{1,2}(\Omega)\hookrightarrow L^2(\Omega)$ is a compact embedding by Rellich's theorem. Therefore, we may apply the Aubin--Lions lemma \cite[Cor. 8.4]{Simon1986}, so that (along a non-relabeled subsequence) $v_{\eps_j}\rightarrow v$ in $L^2(\Omega\times(0,T))$ as $j\to\infty$ follows, which also entails a.e. convergence of a further subsequence. Again, a diagonalization argument concludes the proof. 
\end{proof}

Our aim now is to show that $(u,v)$ is a generalized solution to \eqref{sys} in the sense of definition \ref{def:lsg}. We formulate this in the following theorem, whose proof we will give at the end of this section after additional preparation.

\begin{satz}\label{supersatz}
 Under the assumptions of Theorem \ref{thm:main}, the pair $(u,v)$ obtained above is a generalized solution to system (\ref{sys}).
\end{satz}

By the usual weak compactness arguments and reflexivity of  $L^2(\Omega\times(0,T))$ and $L^2(\Omega)$, the estimates from Section \ref{sec:absch} entail weak convergence of certain terms. This is summarized in the following lemma: 

\begin{lemma}\label{uvschwach} Let $T>0$. 
Under the assumptions of Theorem \ref{thm:main}, the sequences $\uj$ and $\vj$ satisfy:
 \begin{align}
   \label{1}\uj&\rightharpoonup u&&\quad\text{in }L^2(\Omega\times(0,T)),\\
   \label{2}\frac{\uj}{1+\eps_j \uj} &\rightharpoonup u&&\quad\text{in }L^2(\Omega\times(0,T)),\\
   \label{3}\nabla \vj(\cdot,t) &\rightharpoonup \nabla v(\cdot,t)&&\quad\text{in }L^2(\Omega)\text{ for almost all } t\in (0,T),\\
   \label{3.5}\nabla \log(\vj)&\rightharpoonup \nabla \log(v)&&\quad\text{in }L^2(\Omega\times(0,T)),\\
   \label{t} v_{\eps_{j}t}&\rightharpoonup v_t &&\quad\text{in }L^2(\Omega\times(0,T)),\\
    \label{4}\frac{\uj}{(1+\eps_j \uj)(1+\eps_j \vj)}&\rightharpoonup u&&\quad\text{in }L^2(\Omega\times(0,T)),\\
    \label{5}\log(\uj+1) &\rightharpoonup \log(u+1)&&\quad\text{in }L^2(\Omega\times(0,T)),\\
    \label{6}\nabla \log(\uj+1)&\rightharpoonup \nabla \log(u+1)&&\quad\text{in }L^2(\Omega\times(0,T)),\\
    \label{7}\frac{\uj}{(\uj+1)(1+\eps_j\uj)}&\rightharpoonup \frac{u}{u+1}&&\quad\text{in }L^2(\Omega\times(0,T)),\\
    \label{8}\frac{\uj}{(\uj+1)(1+\eps_j\uj)}\nabla \log(\uj+1) &\rightharpoonup \frac{u}{u+1}\nabla \log(u+1)&&\quad\text{in }L^2(\Omega\times(0,T)),\\
    \label{9}\frac{\uj}{\uj+1}&\rightharpoonup \frac{u}{u+1} &&\quad\text{in }L^2(\Omega\times(0,T))\\
    \label{10}\text{and}\quad\frac{\uj^2}{\uj+1}&\rightharpoonup \frac{u^2}{u+1} &&\quad\text{in }L^2(\Omega\times(0,T))
  \end{align}
 as $j\to∞$.
\end{lemma}
\begin{proof}
 According to Lemma \ref{lm2}, $\left\{\ueps\right\}_{\eps\in(0,1)}$ is bounded in $L^2(\Omega\times(0,T))$. \\
 Due to the obvious estimates 
 \begin{align*}
  0&\leq\frac{\ueps}{1+\eps\ueps}\leq\ueps,\quad  
  0&\leq\frac{\ueps}{(1+\eps\ueps)(1+\eps\veps)}\leq \ueps,\quad
  0&\leq\log(\ueps+1)\leq\ueps \quad\text{and }\quad 
  0&\leq\frac{\ueps^2}{\ueps+1}\leq \ueps,
 \end{align*}
 from reflexivity of $L^2(\Omega\times(0,T))$ we conclude the existence of convergent subsequences of the corresponding terms. By Lemma \ref{aubinu}, $\uj\to u$ in $L^1(\Omega\times(0,T))$ and almost everywhere. Since pointwise and weak limit have to coincide if both exist, assertions \eqref{1}, \eqref{2}, \eqref{4}, \eqref{5}, \eqref{10} follow immediately. 
 Analogously, we obtain \eqref{3.5} from Lemma \ref{lm3} and \eqref{t} from Corollary \ref{kor9}, each in conjunction with Lemma \ref{aubinv}.\\

 Suppose, $t\in (0,T)$ were such that \eqref{3} did not hold at $t$. Then we could find a subsequence $\eps_{j_k}$, some $\delta>0$ and some $\varphi\in L^2(\Omega)$
 satisfying 
 \begin{align*}
  \left\lvert \io \nabla v_{\eps_{j_k}}(\cdot,t)\cdot\varphi-\io \nabla v(\cdot,t)\cdot\varphi\right\rvert>\delta
 \end{align*}
 for all $k\in\N$. Due to Lemma \ref{lm7}, however, the sequence $\left(\nabla v_{\eps_{j_k}}(\cdot,t)\right)_{k\in ℕ}$ would have to include an $L^2(\Omega)$-weakly convergent subsequence. For almost all $t\in (0,T)$, Lemma \ref{aubinv} excludes a limit different from $\nabla v(\cdot,t)$.
 Hence, \eqref{3} holds, 
 even without resorting to another subsequence. 
 Assertion (\ref{6}) results from Lemma \ref{lm4} and Lemma \ref{aubinu}, as does (\ref{8}), because 
 \begin{align*}
  \frac{\ueps}{(1+\eps\ueps)(\ueps+1)}\leq 1, 
 \end{align*}
which also shows that $\left\{\frac{\ueps}{(1+\eps\ueps)(\ueps+1)}\right\}_{\eps\in(0,1)}$ is bounded in $L^2(\Omega\times(0,T))$, proving \eqref{7}. Analogously, we obtain \eqref{9} in light of the trivial estimate 
 \(
  \frac{\ueps}{\ueps+1}\leq 1.
 \)
\end{proof}

\subsection{Strong $L^2$-convergence of $\nabla \log(\vj)$}\label{sec:nablalogv}
Assertion \eqref{3.5} can be sharpened in the following sense, which will be decisive for the final proof that the limit object $(u,v)$ is a generalized solution: In the proof of Lemma \ref{sol}, one of the integrals we will have to take to the limit $ε\searrow 0$ will contain the product of the terms in \eqref{3.5} and \eqref{8}. The idea underlying this approach is adapted from \cite{siam}, where in a similar way the $L^2(\Om\times(0,T))$-convergence of $\nabla \vj$ (corresponding to the non-singular sensitivity function in the system considered there) is proven.

\begin{lemma}\label{gradlogv}
 Let $T>0$. Under the assumptions of Theorem \ref{thm:main}, we have $\nabla\log(\vj)\rightarrow\nabla \log(v)$ in $L^2(\Omega\times(0,T))$ as $j\to \infty$.
\end{lemma}

The proof will be based on the idea that for any sequence $(x_n)_{n\in ℕ}$ in a Hilbert space, $x_n\rightharpoonup x$ already implies $x_n\to x$, if at the same time $\Norm{}{x_n}\to \Norm{}{x}$. It will be given at the end of this subsection.

\begin{lemma}\label{log}
 Let $T>0$. Under the assumptions of Theorem \ref{thm:main}, we have $\log (\vj) \rightarrow \log(v)$ in $L^2(\Omega\times(0,T))$ as $j\to\infty$, in particular, there is a subsequence such that $\io \log(\vj(\cdot,t))\rightarrow\io \log(v(\cdot,t))$ as $j\to \infty$ for almost all $t\in(0,T)$.
\end{lemma}
\begin{proof}
 Since $ W^{1,2}(\Omega)\hookrightarrow L^2(\Omega)$ is a compact embedding and $L^2(\Omega) \hookrightarrow (W_0^{2,\infty}(\Omega))^*$ a continuous one, Lemma \ref{lm11} and Lemma \ref{lm_logvt} together with the Aubin--Lions lemma (\cite[Cor. 8.4]{Simon1986}) show that  $\left\{\log(\veps)\right\}_{\eps\in(0,1)}$ is relatively compact in $L^2((0,T);L^2(\Omega))$, hence there is a subsequence with the desired properties.
\end{proof}

\begin{bemerkung}\label{bem:vpos}
 Because $\log(v)\in L^2(\Omega\times(0,T))$, the function $\log(v)$ is finite almost everywhere in $\Omega\times(0,T)$ and hence $v$ is positive almost everywhere. 
\end{bemerkung}

Before we return to dealing with $\log \ve$, let us prepare some more general, technical arguments, on which the proof will rely. These have, for example, not been employed in \cite{siam}.

\begin{lemma}\label{sobolevspaces}
By $W^{1,2}((0,T);L^2(\Om))$ we denote the Sobolev space of square-integrable $L^2(\Om)$-valued functions $u$ on $(0,T)$, whose weak derivative $u_t$ belongs to $L^2((0,T);(L^2(\Om))^*)\cong L^2(\Om\times(0,T))$. We recall that $C^1([0,T];L^2(\Om))$ is dense in $W^{1,2}((0,T);L^2(\Om))$ \cite[23.10 b]{zeidler} and, since this is proven by convolution arguments, it can be seen easily that nonnegative functions in $W^{1,2}((0,T);L^2(\Om))$ can be approximated by a sequence of nonnegative functions in $C^1([0,T];L^2(\Om))$. Moreover, the embedding 
\[
 W^{1,2}((0,T);L^2(\Om)) \hookrightarrow C^0([0,T];L^2(\Om))
\]
is continuous, \cite[Proposition 23.23]{zeidler}. In particular, for every $f\in W^{1,2}((0,T);L^2(\Om))$ there is a function $\tilde{f}\in C^0([0,T];L^2(\Om))$ agreeing with $f$ almost everywhere; and every pointwise evaluation $f(t)$ for some $t\in[0,T]$ is to be understood as $\tilde{f}(t)$.  
\end{lemma}

\begin{lemma}\label{retterindernot}
 Let $\Om\subset ℝ^n$ be a bounded, smooth domain, $T>0$. 
 Let $f\in W^{1,2}((0,T);L^2(\Omega))$ be nonnegative, $\eta>0$. Then
 \begin{align*}
  \int_0^T\!\!\io \frac{f_t f}{(f+\eta)^2}=\io \frac{\eta}{f(\cdot,T)+\eta}+\io\log(f(\cdot,T)+\eta)-\io \frac{\eta}{f(\cdot,0)+\eta}-\io \log(f(\cdot,0)+\eta).
 \end{align*}
\end{lemma}
\begin{proof}
 According to Lemma \ref{sobolevspaces} there is a nonnegative sequence 
 \begin{align*}
 (\varphi_n)_{n\in\N}\subseteq C^{1}([0,T];L^2(\Omega))
 \end{align*} with
 \begin{align*}
  \varphi_n\to f \quad\text{in }W^{1,2}((0,T);L^2(\Omega)) \quad\text{as }n\to \infty.
 \end{align*}
 For every $n\in\N$ by Fubini's theorem and substitution
 \begin{align}
  \int_0^T\!\!\io \frac{\varphi_{n t} \varphi_n}{(\varphi_n+\eta)^2}=&\io\!\int_0^T\frac{\varphi_{n t} \varphi_n}{(\varphi_n+\eta)^2}=\io\!\int_{\varphi_n(\cdot,0)}^{\varphi_n(\cdot,T)} \frac{s}{(s+\eta)^2}\diff s\nn\\
  =&\io \frac{\eta}{\varphi_n(\cdot,T)+\eta}+\io\log(\varphi_n(\cdot,T)+\eta)
  -\io \frac{\eta}{\varphi_n(\cdot,0)+\eta}-\io \log(\varphi_n(\cdot,0)+\eta).\label{eq:10B}
 \end{align}
 Since $\varphi_n\to f$ in $W^{1,2}((0,T);L^2(\Omega))$ as $n\to\infty$, in particular we have $\varphi_{n t}\to f_t$ in $L^2(\Omega\times(0,T))$. 
 Due to 
 \begin{align*}
  \left\lvert \frac{\varphi_{n t} \varphi_n}{(\varphi_n+\eta)^2}\right\rvert\leq \left\lvert \varphi_{n t}\right\rvert\cdot \frac{\varphi_n}{\varphi_n+\eta}\cdot\tel{\varphi_n+\eta}\leq \frac{\left\lvert \varphi_{n t}\right\rvert}{\eta} \qquad \text{ and }\qquad
  \frac{\eta}{\varphi_n+\eta}\leq 1,
 \end{align*}
 from a version of Lebesgue's theorem it follows that 
 \begin{align}\label{eq:10A}
  \lim_{n\to \infty}\int_0^T\!\!\io \frac{\varphi_{n t} \varphi_n}{(\varphi_n+\eta)^2}=\int_0^T\!\!\io \frac{f_t f}{(f+\eta)^2}
 \end{align}
 and 
 \begin{align*}
  \lim_{n\to\infty}\left(\io \frac{\eta}{\varphi_n(\cdot,T)+\eta}-\io \frac{\eta}{\varphi_n(\cdot,0)+\eta}\right)=\io \frac{\eta}{f(\cdot,T)+\eta} -\io \frac{\eta}{f(\cdot,0)+\eta},
 \end{align*}
 because $\varphi_n(\cdot,T)\to f(\cdot,T)$ and $\varphi_n(\cdot,0)\to f(\cdot,0)$ in $L^2(\Omega)$ for $n\to \infty$, since according to Lemma \ref{sobolevspaces} the space $W^{1,2}((0,T);L^2(\Omega))$ is continuously embedded into $C([0,T];L^2(\Omega))$.\\
 For any $t\in[0,T]$, $n\in \N$, by 
\begin{equation}\label{eq:wasloglemma}
 |\log s|\leq 2s-\log s\qquad \text{for all } s>0
\end{equation}
 we have 
 \begin{align*}
  \left\lvert \log(\varphi_n(\cdot,t)+\eta)\right\rvert \leq 2\varphi_n(\cdot,t)+2\eta-\log(\varphi_n(\cdot,t)+\eta)
  \leq 2\varphi_n(\cdot,t)+2\eta-\log(\eta).
 \end{align*}
Thus, another application of Lebesgue's theorem shows that  \begin{align}
  \lim_{n\to\infty}\left(\io\log(\varphi_n(\cdot,T)+\eta)-\io \log(\varphi_n(\cdot,0)+\eta)\right)
  =\io\log(f(\cdot,T)+\eta)-\io \log(f(\cdot,0)+\eta),\label{eq:10C}
 \end{align}
and \eqref{eq:10A}, \eqref{eq:10B} and \eqref{eq:10C} taken together imply 
 \begin{equation*}
  \int_0^T\!\!\io \frac{f_t f}{(f+\eta)^2}=\io \frac{\eta}{f(\cdot,T)+\eta}+\io\log(f(\cdot,T)+\eta) 
 -\io \frac{\eta}{f(\cdot,0)+\eta}-\io \log(f(\cdot,0)+\eta).\qedhere
 \end{equation*}
\end{proof}

Additionally, we will use the following chain rule for Sobolev functions, which we recall briefly:

\begin{lemma}\label{kettenregel}
 Let $\Omega\subsetℝ^n$ be a domain, let $f\colon\R\to\R$ be Lipschitz continuous and let $y\in W^{1,p}(\Omega)$ for some $p\geq1$. If $f\circ y\in L^p(\Omega)$, then $f\circ y\in W^{1,p}(\Omega)$ and for almost every $x\in\Omega$:
 \begin{align*}
  \nabla (f\circ y)(x)=f'(y(x))\nabla y(x).
 \end{align*}
\end{lemma}
\begin{proof}
 See \cite[Theorem 2.1.11]{ziemer}.
\end{proof}

All of these preparations will now be taken to their use in the proof of the following lemma:

\begin{lemma}\label{limsup}
 Under the assumptions of Theorem \ref{thm:main}, for almost all $T>0$
 \begin{align*}
  \lim_{\eps_j\searrow0} \int_0^T\!\!\io |\nabla \log(\vj)|^2= \int_0^T\!\!\io |\nabla\log(v)|^2.
  \end{align*}
\end{lemma}
\begin{proof}
 Testing the equation for $v_{\eps t}$  in \eqref{syseps} by $\varphi$, for every $\eps>0$ by integration by parts we obtain 
 \begin{align}\label{eq:phieps}
  -\int_0^T\!\!\io \varphi_t\veps+\io \varphi(\cdot,T)\veps(\cdot,T)-\io \varphi(\cdot, 0)v_0
  =-\int_0^T\!\!\io \nabla \veps \cdot \nabla\varphi -\int_0^T\!\!\io \frac{\ueps\veps}{(1+\eps\ueps)(1+\eps\veps)}\varphi.
 \end{align}
Here we pass to the limit along the sequence $\eps_j\searrow 0$, employing Lemma \ref{aubinv} on the left hand side of \eqref{eq:phieps} and \eqref{3} or a combination of \eqref{4} and, again, Lemma \ref{aubinv}, respectively, in the integrals on the right. This shows that for every $\varphi\in W^{1,2}(\Omega\times(0,T))$: 
 \begin{align}\label{eq:phi}
  -\int_0^T\!\!\io \varphi_t v +\io v(\cdot,T)\varphi(\cdot,T)-\io v_0\varphi(\cdot,0)=-\int_0^T\!\!\io \nabla v \cdot \nabla \varphi -\int_0^T\!\!\io \varphi uv.
  \end{align}
For every $\eta>0$ we now define 
 \begin{align*}
  \varphi_{\eta}(x,t):=\frac{1}{v(x,t)+\eta}.
 \end{align*}
The map $[0,∞)\ni s\mapsto \tel{s+\eta}$ is Lipschitz continuous and $\varphi_{\eta}\in L^2(\Om\times(0,T))$ due to 
 \begin{align*}
 \frac{1}{v(x,t)+\eta}\leq\tel{\eta},
 \end{align*}
 and we may apply Lemma \ref{kettenregel}, since ${v\in W^{1,2}(\Omega\times(0,T))}$ by Lemma \ref{aubinv} together with (\ref{3}) and (\ref{t}) from Lemma \ref{uvschwach}. \\
 Hence $\varphi_{\eta}\in W^{1,2}(\Omega\times(0,T))$ and can be inserted into (\ref{eq:phi}) in place of $\varphi$. By Lemma  \ref{kettenregel}, $\nabla \varphi_{\eta}\cdot\nabla v=-\frac{\nabla v}{(v+\eta)^2}\cdot\nabla v$ 
 and $\varphi_{\eta t}=-\frac{v_t}{(v+\eta)^2}$ -- and thus (\ref{eq:phi}) with $\varphi_{\eta}$ turns into:
 \begin{align}\label{eq:limsuplemma}
  \int_0^T\!\!\io \frac{|\nabla v|^2}{(v+\eta)^2}=\int_0^T\!\!\io \frac{v_t v}{(v+\eta)^2}+\io \frac{v(\cdot,T)}{v(\cdot,T)+\eta} 
  -\io \frac{v_0}{v_0+\eta}+\int_0^T\!\!\io \frac{uv}{v+\eta}.
 \end{align}
Beppo Levi's theorem shows that 
 \begin{align*}
  \lim_{\eta\searrow 0} \int_0^T\!\!\io \frac{|\nabla v|^2}{(v+\eta)^2}&=\int_0^T\!\!\io \frac{|\nabla v|^2}{v^2}=\int_0^T\!\!\io|\nabla \log(v)|^2,\quad
  \lim_{\eta\searrow 0}\io \frac{v(\cdot,T)}{v(\cdot,T)+\eta}=|\Omega|= \lim_{\eta\searrow 0}\io \frac{v_0}{v_0+\eta},\\
  &\text{ and }\lim_{\eta\searrow 0}\int_0^T\!\!\io \frac{uv}{v+\eta}=\int_0^T\!\!\io u.
 \end{align*}
 Thanks to (\ref{t}) and Lemma \ref{aubinv}, moreover $v\in W^{1,2}((0,T);L^2(\Omega))$. Furthermore, according to Remark \ref{bem:vpos}, $v$ is positive almost everywhere. Lemma \ref{retterindernot} hence implies that 
 \begin{align*}
  \int_0^T\!\!\io \frac{v_t v}{(v+\eta)^2}=&\io \frac{\eta}{v(\cdot,T)+\eta}+\io\log(v(\cdot,T)+\eta)
  -\io \frac{\eta}{v(\cdot,0)+\eta}-\io \log(v(\cdot,0)+\eta).
 \end{align*}
Due to \eqref{eq:wasloglemma}, 
 for every $t\in[0,T]$, $\eta>0$,  
 \begin{align*}
 \left\lvert\frac{\eta}{v(\cdot,t)+\eta}\right\rvert\leq 1 \quad\text{and } \lvert\log(v(\cdot,t)+\eta)\rvert\leq 2v(\cdot,t)+2\eta-\log(v(\cdot,t)).
\end{align*}
 Because $v(\cdot,t)+2\eta-\log(v(\cdot,t))\in L^2(\Omega)$ by Lemma \ref{aubinv} and Lemma \ref{log}, Lebesgue's theorem is applicable.
 As 
 \begin{align*}
  \lim_{\eta\searrow0}\frac{\eta}{v(\cdot,t)+\eta}=0\quad\text{and } \lim_{\eta\searrow0}\log(v(\cdot,t)+\eta)=\log(v(\cdot,t))
 \end{align*}
 for $t=T$ and $t=0$, in passing to the limit $\eta \searrow0$ it implies 
\begin{align*}
 \io \frac{\eta}{v(\cdot,T)+\eta}&+\io\log(v(\cdot,T)+\eta)-\io \frac{\eta}{v(\cdot,0)+\eta}-\io \log(v(\cdot,0)+\eta)
  \rightarrow \io \log(v(\cdot,T))-\io \log(v_0).
\end{align*}
 Hence in (\ref{eq:limsuplemma}) we obtain:
 \begin{align*}
  \int_0^T\!\!\io |\nabla\log(v)|^2= \io \log(v(\cdot,T))-\io\log(v_0)+\int_0^T\!\!\io u.
 \end{align*}
 Furthermore, from (\ref{eq:phieps}) with $\varphi=\tel{\veps}$ for every $\eps>0$ we have 
 \begin{align*}
  \int_0^T\!\!\io |\nabla \log(\veps)|^2=\io \log(\veps(\cdot,T))-\io\log(v_0)+\int_0^T\!\!\io \frac{\ueps}{(1+\eps\ueps)(1+\eps\veps)}.
 \end{align*}
 For almost every $T$ we have $\lim_{\eps_j\searrow 0}\io \log(v_{\eps_j}(\cdot,T))=\io \log(v(\cdot,T))$ by Lemma \ref{log}. As moreover $\lim_{\eps_j\searrow 0}\int_0^T\!\!\io \frac{\uj}{(1+\eps_j\uj)(1+\eps_j\vj)}=\int_0^T\!\!\io u$, in conclusion we arrive at 
 \begin{align*}
  \int_0^T\!\!\io |\nabla\log(v)|^2 =& \io \log(v(\cdot,T))-\io\log(v_0)+\int_0^T\!\!\io u\\
  =&\lim_{\eps_j\searrow0} \io \log(v_{\eps_j}(\cdot,T)) -\io\log(v_0)+\lim_{\eps_j\searrow0} \int_0^T\!\!\io \frac{\uj}{(1+\eps_j\uj)(1+\eps_j\vj)}\\
  =&\lim_{\eps_j\searrow0} \int_0^T\!\!\io |\nabla\log(v_{\eps_j})|^2.\qedhere
 \end{align*}
\end{proof}

Thus, we now can prove strong convergence of $\nabla\log(\vj)$ in $L^2(\Omega\times(0,T))$:
\begin{proof}[Proof of Lemma \ref{gradlogv}]
Due to Lemma \ref{limsup} we have 
 \begin{align*}
  \int_0^T\!\!\io |\nabla\log(v)|^2=\lim_{\eps_j\searrow0} \int_0^T\!\!\io |\nabla \log(\vj)|^2,
  \end{align*}
  that is 
  \begin{align*}
  \Norm{L^2(\Omega\times(0,T))}{\nabla \log(\vj)}\rightarrow \Norm{L^2(\Omega\times(0,T))}{\nabla \log(v)}\quad\text{as }\eps_j\searrow0.
 \end{align*}
As mentioned in the beginning of this subsection, together with the weak convergence asserted by \eqref{3.5} this is sufficient to infer 
 \begin{align*}
  \nabla\log(\vj)\rightarrow\nabla \log(v)\quad\text{in }L^2(\Omega\times(0,T)) \qquad \text{as } \eps_j\searrow 0,
 \end{align*}
which concludes the proof. 
\end{proof}

\subsection{$(u,v)$ is a solution}
After these preparations, we can now show that $(u,v)$ is a generalized solution.

\begin{lemma}\label{lsgv}
 Under the assumptions of Theorem \ref{thm:main}, the pair $(u,v)$ satisfies 
 \begin{align*}
  -\int_0^{\infty}\!\!\io \psi_tv-\io v_0\psi(\cdot,0)=-\int_0^{\infty}\!\!\io \nabla v \cdot \nabla \psi-\int_0^{\infty}\!\!\io \psi uv
 \end{align*}
for every $\psi\in C_0^{\infty}(\overline{\Omega}\times [0,\infty))$.
\end{lemma}
\begin{proof}
 Let $\psi\in C_0^{\infty}(\overline{\Omega}\times [0,\infty))$. Then for every $\eps>0$,  
 \begin{align*}
    -\int_0^{\infty}\!\!\io \psi_t\veps-\io v_0\psi(\cdot,0)=-\int_0^{\infty}\!\!\io \nabla \veps\cdot\nabla \psi-\int_0^{\infty}\!\!\io \psi \frac{\ueps\veps}{(1+\eps\ueps)(1+\eps\veps)}.
 \end{align*}
 Due to $\psi$ being compactly supported, there is $T>0$ such that $\psi(\cdot,t)=0$ for all $t>T$. Due to Lemma \ref{aubinv}, 
 \begin{align*}
  -\int_0^{\infty}\!\!\io \psi_t\vj = -\int_0^{T}\io \psi_t\vj\rightarrow -\int_0^{T}\io \psi_tv= -\int_0^{\infty}\!\!\io \psi_tv
 \end{align*}
 converges as $\eps_j\searrow0$. 
 Assertion (\ref{3}) of Lemma \ref{uvschwach} implies 
 \begin{align*}
  -&\int_0^{\infty}\!\!\io \nabla \vj\nabla \psi \rightarrow -\int_0^{\infty}\!\!\io \nabla v \cdot \nabla \psi \qquad \text{as } \eps_j\searrow 0.
 \end{align*}
 Lemma \ref{aubinv} and (\ref{4}) entail that moreover 
 \begin{align*}
  -\int_0^{\infty}\!\!\io \psi \frac{\uj\vj}{(1+\eps_j\uj)(1+\eps_j\vj)}\rightarrow -\int_0^{\infty}\!\!\io \psi uv \qquad \text{as } \eps_j\searrow 0.
 \end{align*}
Therefore, we obtain 
 \begin{align*}
  -\int_0^{\infty}\!\!\io \psi_tv-\io v_0\psi(\cdot,0)&\leftarrow -\int_0^{\infty}\!\!\io \psi_t\vj-\io v_0\psi(\cdot,0)\\
  &=-\int_0^{\infty}\!\!\io \nabla \vj\nabla \psi-\int_0^{\infty}\!\!\io \psi \frac{\uj\vj}{(1+\eps_j\uj)(1+\eps_j\vj)}\\
  &\rightarrow -\int_0^{\infty}\!\!\io \nabla v \cdot \nabla \psi-\int_0^{\infty}\!\!\io \psi uv\quad\text{as }\eps_j\searrow0
 \end{align*}
 and hence the claim, due to uniqueness of the limit. 
\end{proof}

The solution property concerning the second equation of \eqref{sys} is hence satisfied. We still have to deal with the first equation. 

\begin{lemma}\label{sul}
 Under the assumptions of Theorem \ref{thm:main}, the pair $(u,v)$ is a very weak subsolution to system (\ref{sys}).
\end{lemma}
\begin{proof}
  As $\ueps\geq 0$ for every $\eps>0$ and $\uj\rightarrow u$ converges almost everywhere as $ε_j\searrow 0$, $u$ is nonnegative. According to Remark \ref{bem:vpos}, $v$ is positive almost everywhere. 
  From (\ref{1}) in Lemma \ref{uvschwach} we obtain \mbox{$u\in L^2_{\text{loc}}([0,\infty);L^2(\Omega))$}.  
  Lemma \ref{lm:veps} and assertion (\ref{3}) together yield $v\in L^2_{\text{loc}}([0,\infty);W^{1,2}(\Omega))$, and Lemma \ref{gradlogv} implies that \mbox{$\nabla\log(v)\in L^2_{\text{loc}}(\Ombar\times[0,\infty))$}. \\
  We now let $\varphi\in C_0^{\infty}(\overline{\Omega}\times [0,\infty))$ be nonnegative with $\partial_\nu \varphi=0$ on  $\partial\Omega$. Then for every  $\eps>0$ we have 
  \begin{align*}
  -\int_0^{\infty}\!\!\io \varphi_t\ueps-\io u_0\varphi(\cdot,0)=& \int_0^{\infty}\!\!\io \uepsΔ\varphi+\chi\int_0^{\infty}\!\!\io \frac{\ueps}{1+\eps\ueps}\nabla\varphi\cdot\nabla\log(\veps)\\
  &+\kappa\int_0^{\infty}\!\!\io \ueps\varphi-\mu\int_0^{\infty}\!\!\io \varphi \ueps^2.
 \end{align*}
 Because of (\ref{1}), the following integrals converge:
 \begin{align*}
  -\int_0^{\infty}\!\!\io \varphi_t\uj&\rightarrow -\int_0^{\infty}\!\!\io \varphi_t u,\\
  \int_0^{\infty}\!\!\io \ujΔ\varphi&\rightarrow \int_0^{\infty}\!\!\io uΔ\varphi\quad\text{and}\\
  \kappa\int_0^{\infty}\!\!\io \uj\varphi&\rightarrow \kappa\int_0^{\infty}\!\!\io u\varphi\quad\text{as }\eps_j\searrow0.
 \end{align*}
 In combination with (\ref{2}), Lemma \ref{gradlogv} moreover implies that 
 \begin{align*}
  \chi\int_0^{\infty}\!\!\io \frac{\uj}{1+\eps_j\uj}\nabla\varphi\cdot\nabla\log(\vj) \rightarrow \chi\int_0^{\infty}\!\!\io u\nabla\varphi\cdot\nabla\log(v) \quad\text{as }\eps_j\searrow0
 \end{align*}
holds. Since according to Lemma \ref{aubinu} $\varphi\uj^2\rightarrow \varphi u^2$ converges a.e., Fatou's lemma due to the nonnegativity of  $\varphi \uj^2$ shows that 
 \begin{align*}
  \int_0^{\infty}\!\!\io \varphi u^2\leq \liminf_{\eps_j\searrow0} \int_0^{\infty}\!\!\io \varphi \uj^2,
 \end{align*}
 and hence 
 \begin{align*}
    -\mu\int_0^{\infty}\!\!\io \varphi u^2\geq -\liminf_{\eps_j\searrow0} \mu \int_0^{\infty}\!\!\io \varphi \uj^2.
 \end{align*}
 In conclusion, we obtain
 \begin{align*}
  -\int_0^{\infty}\!\!\io \varphi_tu-\io u_0\varphi(\cdot,0)=&\lim_{\eps_j\searrow 0} \left(-\int_0^{\infty}\!\!\io \varphi_t\uj\right)-\io u_0\varphi(\cdot,0)\\
  =&\lim_{\eps_j\searrow 0} \bigg(\int_0^{\infty}\!\!\io \ujΔ\varphi+\chi\int_0^{\infty}\!\!\io \frac{\uj}{1+\eps_j\uj}\nabla\varphi\cdot\nabla\log(\vj)\\
  &+\kappa\int_0^{\infty}\!\!\io \uj\varphi-\mu\int_0^{\infty}\!\!\io \varphi \uj^2 \bigg)\\
  \leq& \int_0^{\infty}\!\!\io uΔ\varphi+\chi\int_0^{\infty}\!\!\io u\nabla\varphi\cdot\nabla\log(v)+\kappa\int_0^{\infty}\!\!\io u\varphi-\mu\int_0^{\infty}\!\!\io \varphi u^2.
 \end{align*}
 From Lemma \ref{lsgv} we therefore can conclude that $(u,v)$ is a very weak subsolution to (\ref{sys}).
\end{proof}

\begin{lemma}\label{sol}
 Under the assumptions of Theorem \ref{thm:main}, the pair $(u,v)$ is a weak logarithmic supersolution to \eqref{sys}. 
\end{lemma}
\begin{proof}
 In the previous proof we already noted that $u$ is nonnegative and $v$ is positive almost everywhere and that ${v\in L^2_{\text{loc}}([0,\infty);W^{1,2}(\Omega))}$ and ${\nabla\log(v)\in L^2_{\text{loc}}(\Ombar\times[0,\infty))}$ hold. Moreover, $u\in L^2_{\text{loc}}([0,\infty);L^2(\Omega))$ and thus, according to Hölder's inequality also ${u\in L^1_{\text{loc}}([0,\infty);L^2(\Omega))}$. \\By Lemma \ref{lm:veps}, furthermore, $v\in L^{\infty}_{\text{loc}}(\Ombar\times[0,\infty))$. In addition, Lemma \ref{uvschwach} and assertion (\ref{6}) imply $\nabla\log(u+1)\in L^2_{\text{loc}}([0,\infty);L^2(\Omega))$.\\

 We now let $\varphi\in C_0^{\infty}(\overline{\Omega}\times [0,\infty))$ be a nonnegative function satisfying $\partial_\nu \varphi=0$ on  $\partial\Omega$. 
Testing the first equation of \eqref{syseps} by $\frac{\varphi}{\ueps+1}$, for every $ε>0$ we obtain via integration by parts
 \begin{align*}
  -\int_0^{\infty}\io \log(\ueps+1)\varphi_t-&\io \log(u_0+1)\varphi(\cdot,0)\\
  =&-\int_0^{\infty}\!\!\io \nabla\log(\ueps+1)\cdot\nabla\varphi + \int_0^{\infty}\!\!\io |\nabla\log(\ueps+1)|^2\varphi\\
  &+ \chi \int_0^{\infty}\!\!\io \frac{\ueps}{(1+\eps\ueps)(\ueps+1)}\nabla\log(\veps)\cdot\nabla \varphi \\
  &- \chi \int_0^{\infty}\!\!\io \frac{\ueps}{(1+\eps\ueps)(\ueps+1)}\varphi\nabla\log(\ueps+1)\cdot\nabla\log(\veps) \\
  &+\kappa \int_0^{\infty}\!\!\io \frac{\ueps}{\ueps+1}\varphi-\mu \int_0^{\infty}\!\!\io \frac{\ueps^2}{\ueps+1}\varphi.
 \end{align*}
 According to Lemma \ref{uvschwach}, (\ref{5}) and (\ref{6}), the terms 
 \begin{align*}
  \int_0^{\infty}\!\!\io \log(\uj+1)\varphi_t&\rightarrow \int_0^{\infty}\!\!\io \log(u+1)\varphi_t
  \end{align*}
  and 
  \begin{align*}
  \int_0^{\infty}\!\!\io \nabla\log(\uj+1)\cdot\nabla\varphi&\rightarrow \int_0^{\infty}\!\!\io \nabla\log(u+1)\nabla\varphi
 \end{align*}
 converge as $\eps_j\searrow0$. From (\ref{7}) and Lemma \ref{gradlogv} we may conclude that 
 \begin{align*}
  \int_0^{\infty}\!\!\io \frac{\uj}{(1+\eps_j\uj)(\uj+1)}\nabla\log(\vj)\cdot\nabla \varphi \rightarrow \int_0^{\infty}\!\!\io \frac{u}{u+1}\nabla\log(v)\cdot\nabla \varphi 
 \end{align*}
 as $\eps_j\searrow0$. 
 Analogously, (\ref{8}) of Lemma \ref{uvschwach} together with Lemma \ref{gradlogv} entails the convergence
 \begin{align*}
  \int_0^{\infty}\!\!\io \frac{\uj}{(1+\eps_j\uj)(\uj+1)}\varphi\nabla\log(\uj+1)\cdot\nabla\log(\vj) \\
  \rightarrow \int_0^{\infty}\!\!\io \frac{u}{u+1}\varphi\nabla\log(u+1)\cdot\nabla\log(v)\quad\text{as }\eps_j\searrow0.
 \end{align*}
Assertions (\ref{9}) and (\ref{10}) from Lemma \ref{uvschwach} moreover show that 
 \begin{align*}
  \int_0^{\infty}\!\!\io \frac{\uj}{\uj+1}\varphi &\rightarrow \int_0^{\infty}\!\!\io \frac{u}{u+1}\varphi \quad\text{and}\\
  \int_0^{\infty}\!\!\io \frac{\uj^2}{\uj+1}\varphi &\rightarrow \int_0^{\infty}\!\!\io \frac{u^2}{u+1}\varphi\quad\text{as }\eps_j\searrow0.
 \end{align*}
 Finally, \eqref{6} and boundedness and nonnegativity of $\sqrt{φ}$ yield 
 \begin{align*}
  \sqrt{\varphi}\nabla\log(\uj+1)\rightharpoonup\sqrt{\varphi}\log(u+1)\quad\text{in }L^2(\Omega\times(0,T))
 \end{align*}
and hence by weak sequential lower semicontinuity of the $L^2$-norm we obtain
 \begin{align*}
  \int_0^{\infty}\!\!\io \varphi |\nabla\log(u+1)|^2\leq\liminf_{\eps_j\searrow0}\int_0^{\infty}\!\!\io \varphi |\nabla\log(\uj+1)|^2.
 \end{align*}
In conclusion, this has shown that 
 \begin{align*}
  -\int_0^{\infty}\io \log(u+1)\varphi_t&-\io \log(u_0+1)\varphi(\cdot,0)\\
  =&\lim_{\eps_j\searrow 0}\left( -\int_0^{\infty}\!\!\io \log(\uj+1)\varphi_t-\io \log(u_0+1)\varphi(\cdot,0)\right)\\
  =&\lim_{\eps_j\searrow 0}\bigg(-\int_0^{\infty}\!\!\io \nabla\log(\uj+1)\cdot\nabla\varphi + \int_0^{\infty}\!\!\io |\nabla\log(\uj+1)|^2\varphi\\
  &+ \chi \int_0^{\infty}\!\!\io \frac{\uj}{(1+\eps_j\uj)(\uj+1)}\nabla\log(\vj)\cdot\nabla \varphi \\
  &- \chi \int_0^{\infty}\!\!\io \frac{\uj}{(1+\eps_j\uj)(\uj+1)}\varphi\nabla\log(\uj+1)\cdot\nabla\log(\vj) \\
  &+\kappa \int_0^{\infty}\!\!\io \frac{\uj}{\uj+1}\varphi-\mu \int_0^{\infty}\!\!\io \frac{\uj^2}{\uj+1}\varphi \bigg)\\
  \geq& -\int_0^{\infty}\!\!\io\nabla\log(u+1)\cdot\nabla\varphi+\int_0^{\infty}\!\!\io\varphi|\nabla\log(u+1)|^2\\
  &+\chi\int_0^{\infty}\!\!\io\frac{u}{u+1}\nabla \log(v) \cdot \nabla\varphi\\
  &-\chi \int_0^{\infty}\!\!\io \frac{u}{u+1}\varphi \nabla\log(v)\cdot\nabla \log(u+1) \\
  &+ \kappa\int_0^{\infty}\!\!\io\frac{u}{u+1}\varphi-\mu\int_0^{\infty}\!\!\io\frac{u^2}{u+1}\varphi,
 \end{align*}
wherefore the claim results from Lemma \ref{lsgv}.
\end{proof}

Thus the existence of a global generalized solution is proven: 

\begin{proof}[Proof of Theorem \ref{supersatz}]
 By definition \ref{def:lsg} of generalized solutions, from Lemma \ref{sul} and Lemma \ref{sol} we immediately have that $(u,v)$ is a generalized solution to \eqref{sys}. 
\end{proof}
\begin{proof}[Proof of Theorem \ref{thm:main}]
 According to Theorem \ref{supersatz}, the functions obtained in Lemma \ref{aubinu} and Lemma \ref{aubinv} are a global generalized solution. This proves the existence of a global generalized solution. 
\end{proof}

{\footnotesize

}

\end{document}